\documentclass[11pt]{amsart}
\usepackage[utf8]{inputenc}
\usepackage{amsmath}
\usepackage{stmaryrd}
\usepackage{MnSymbol}
\usepackage{hyperref}
\usepackage{tikz}
\usetikzlibrary{automata}
\newtheorem{thm}{Theorem}[section]

\newtheorem{lem}[thm]{Lemma}
\newtheorem{prop}[thm]{Proposition}
\newtheorem{example}[thm]{Example}

\newtheorem{opm}[thm]{Open Problem}
\numberwithin{equation}{section}
\newcommand{\abs}[1]{\left\vert#1\right\vert}

\def\malcev{\mathop{\raise1pt\hbox{\footnotesize$\bigcirc$\kern-8pt\raise1pt\hbox{\tiny$m$}\kern1pt}}}
\begin{document}
\title{Nilpotency and strong nilpotency for finite semigroups}
\author{J. Almeida, M. Kufleitner and M. H. Shahzamanian}
\address{J. Almeida and M. H. Shahzamanian\\ Centro de Matemática e Departamento de Matemática, Faculdade de Ciências,
Universidade do Porto, Rua do Campo Alegre, 687, 4169-007 Porto,
Portugal}
\email{jalmeida@fc.up.pt; m.h.shahzamanian@fc.up.pt\footnote{Corresponding author}}
\address{M. Kufleitner\\ Formal Methods in Computer Science (FMI),  University of Stuttgart, Universitätsstr. 38, D-70569 Stuttgart, Germany}
\email{kufleitner@fmi.uni-stuttgart.de}
\thanks{ 2010 Mathematics Subject Classification. Primary 20M07, 20M35.\\
Keywords and phrases: profinite topologies, monoids, block groups, pseudovarieties, nilpotent semigroups in the sense of Mal'cev.}

\begin{abstract}
Nilpotent semigroups in the sense of Mal'cev are defined by semigroup identities. Finite nilpotent semigroups constitute a pseudovariety, $\mathsf{MN}$, which has finite rank. The semigroup identities that define nilpotent semigroups, lead us to define strongly Mal'cev nilpotent semigroups. Finite strongly Mal'cev nilpotent semigroups constitute a non-finite rank pseudovariety, $\mathsf{SMN}$. The pseudovariety $\mathsf{SMN}$ is strictly contained in the pseudovariety $\mathsf{MN}$ but all finite nilpotent groups are in $\mathsf{SMN}$. 
We show that the pseudovariety $\mathsf{MN}$ is the intersection of the pseudovariety $\mathsf{BG_{nil}}$ with a pseudovariety defined by a $\kappa$-identity.
We further compare the pseudovarieties $\mathsf{MN}$ and $\mathsf{SMN}$ with the Mal'cev product $\mathsf{J} \malcev \mathsf{G_{nil}}$. 
\end{abstract}
\maketitle

\tableofcontents

\section{Introduction}\label{pre}
Mal'cev \cite{Mal} and independently Neumann and Taylor \cite{Neu-Tay} have shown that nilpotent groups can be defined by semigroup identities (that is, without using inverses). This leads to the notion of a nilpotent semigroup (in the sense of Mal'cev). 

For a semigroup $S$ with elements $x,y,z_{1},z_{2}, \ldots$ one recursively defines two sequences 
$$\lambda_n=\lambda_{n}(x,y,z_{1},\ldots, z_{n})\quad{\rm and} \quad \rho_n=\rho_{n}(x,y,z_{1},\ldots, z_{n})$$ by
$$\lambda_{0}=x, \quad \rho_{0}=y$$ and
$$\lambda_{n+1}=\lambda_{n} z_{n+1} \rho_{n}, \quad \rho_{n+1}=\rho_{n} z_{n+1} \lambda_{n}.$$ A $S$ semigroup is said to be \emph{nilpotent} if there exists a positive integer $n$ such that 
$$\lambda_{n}(x,y,z_{1},\ldots, z_{n}) = \rho_{n}(x,y,z_{1},\ldots, z_{n})$$ for all $x,y$ in $S$ and $z_{1}, \ldots, z_{n}$
in $S^{1}$. 
The smallest such $n$ is called the \emph{nilpotency class}
of $S$. Clearly, null semigroups are nilpotent in the sense of Mal'cev.  

A pseudovariety of semigroups is a class of finite semigroups closed under taking subsemigroups, homomorphic images and finite direct products. The finite nilpotent semigroups constitute a pseudovariety which is denoted by $\mathsf{MN}$ \cite{Sha}. 
In \cite{Al-Sha}, the rank of the pseudovariety $\mathsf{MN}$ and some classes defined by several of the variants of Mal'cev nilpotent semigroups are investigated and they are compared. 

Let $S$ be a semigroup.
In this paper, we introduce a further variant of Mal'cev nilpotency, that we call strong Mal'cev nilpotency. For semigroups, the new notion is strictly stronger than Mal'cev nilpotency, but it coincides with nilpotency for groups.
Strongly Mal'cev nilpotent semigroups constitute a pseudovariety which we denote by $\mathsf{SMN}$. We show that $\mathsf{G_{nil}}\subsetneqq\mathsf{SMN}\subsetneqq \mathsf{MN}$ where $\mathsf{G_{nil}}$ is the pseudovariety of all finite nilpotent groups.
Higgins and Margolis showed that $\langle\mathsf{A}\cap\mathsf{Inv}\rangle\subsetneqq \mathsf{A}\cap\langle\mathsf{Inv}\rangle$ \cite{Hig-Mar}. 
In \cite{Al-Sha}, it is proved that $\langle\mathsf{A}\cap\mathsf{Inv}\rangle \subsetneqq\mathsf{A}\cap\mathsf{MN}$. We show that, in fact, $\langle\mathsf{A}\cap\mathsf{Inv}\rangle \subsetneqq\mathsf{A}\cap\mathsf{SMN}$.

The paper \cite{Al-Sha} also shows that
$\mathsf{MN}$ is defined by the pseudoidentity $\phi^{\omega}(x)=\phi^{\omega}(y)$,
where $\phi$ is the continuous endomorphism of the free profinite semigroup on $\{x,y,z,t\}$ such that $\phi(x)=xzytyzx$, $\phi(y)=yzxtxzy$, $\phi(z)=z$, and $\phi(t)=t$. 
In particular, the pseudovariety $\mathsf{MN}$ has finite rank.
We prove that the pseudovariety $\mathsf{SMN}$ has infinite rank and, therefore, it is non-finitely based.
In this paper, we also show that the pseudovariety $\mathsf{MN}$ is the intersection of $\mathsf{BG_{nil}}$ with a pseudovariety defined by a $\kappa$-identity.

Note that the following chain of proper inclusions holds: 
$$\mathsf{G_{nil}}\subsetneqq\mathsf{SMN}\subsetneqq \mathsf{MN}\subsetneqq\mathsf{BG}_{nil}.$$
On the other hand, it is part of a celebrated result that 
$\mathsf{BG}=\mathsf{J} \malcev \mathsf{G}$
where $\malcev$ stands for Mal'cev product \cite{Pin2}. In contrast, the inclusion 
$\mathsf{J} \malcev \mathsf{H}\subsetneqq \mathsf{BH}$ is strict for every proper subpseudovariety $\mathsf{H}$ of $\mathsf{G}$ \cite{Hig-Mar}.

\section{Preliminaries}\label{sec:prelims}
For standard notation and terminology relating to finite semigroups, we refer the reader to
\cite{Cli}. A completely $0$-simple finite semigroup $S$ is isomorphic
with a regular Rees matrix semigroup $\mathcal{M}^{0}(G, n,m;P)$,
where $G$ is a maximal subgroup of $S$, $P$ is the $m\times n$
sandwich matrix with entries in $G^{\theta}$ and $n$ and $m$ are
positive integers. The nonzero elements of $S$ are denoted 
$(g;i,j)$, where $g\in G$, $1\leq i \leq n$ and $1\leq j\leq m$; the
zero element is denoted $\theta$. The $(j,i)$-entry of $P$ is denoted $p_{ji}$. The set of nonzero elements
is denoted $\mathcal{M} (G,n,m;P)$. If all elements of $P$ are
nonzero then $\mathcal{M} (G,n,m;P)$ is a subsemigroup and every completely simple finite
semigroup is of this form. If $P=I_{n}$, the $n\times n$ identity
matrix, then $S$ is an inverse semigroup. Jespers and Okni{\'n}ski proved
that a completely $0$-simple semigroup $\mathcal{M}^{0}(G,n,m;P)$ is
Mal'cev nilpotent if and only if $n=m$, $P=I_{n}$ and $G$ is a nilpotent group
\cite[Lemma 2.1]{Jes-Okn}.

The next lemma is a necessary and sufficient condition for a finite semigroup not to be nilpotent \cite[Lemma 2.2]{Jes-Sha}.
\begin{lem} \label{finite-nilpotent}
A finite semigroup $S$ is not Mal'cev nilpotent if and only if there exist a positive integer $m$, distinct elements $x, y\in S$, and elements 
$ w_{1}, w_{2}, \ldots, w_{m}\in S^{1}$ such that $$x = \lambda_{m}(x, y, w_{1}, w_{2}, \ldots, w_{m}) \mbox{ and } y = \rho_{m}(x,y, w_{1}, w_{2}, \ldots, w_{m}).$$
\end{lem}

Assume that a finite semigroup $S$ has a proper ideal $M=\mathcal{M}^{0}(G,n,n;I_{n})$ and $n>1$.
The action $\Gamma$ on the ${\mathcal R}$-classes of $M$ in \cite{Jes-Sha3} is used. In this paper, we consider the dual definition of the action $\Gamma$ as in \cite{Al-Sha}. The action $\Gamma$ is defined to be the action
of $S$ on the ${\mathcal L}$-classes of $M$, that is a representation (a semigroup homomorphism)
$\Gamma : S\to \mathcal{T},$
where $\mathcal{T}$ denotes the full transformation semigroup on the set $\{1, \ldots, n\} \cup \{\theta\}$.
The definition is as follows, for $1\leq j\leq n$ and $s\in S$,
   $$\Gamma(s)(j) = \left\{ \begin{array}{ll}
      j' & \mbox{if} ~(g;i,j)s=(g';i,j')  ~ \mbox{for some} ~ g, g' \in G, \, 1\leq i \leq n\\
       \theta & \mbox{otherwise}\end{array} \right.$$
and
      $\Gamma (s)(\theta ) =\theta .$
We call the representation $\Gamma$ the \emph{${\mathcal L}_M$-representation} of $S$.

For every $s \in S$, $\Gamma(s)$ can be written as a product of orbits which are cycles of the form
$(j_{1}, j_{2}, \ldots, j_{k})$  or sequences of the form $(j_{1}, j_{2},
\ldots, j_{k}, \theta )$, where $1\leq j_{1}, \ldots, j_{k}\leq n$.
The latter orbit means that $\Gamma(s)(j_{i})=j_{i+1}$ for $1\leq
i\leq  k-1$, $\Gamma(s)(j_{k})=\theta$, $\Gamma(s)(\theta) =\theta$ and
there does not exist $1\leq r \leq n$ such that  $\Gamma
(s)(r)=j_{1}$. Orbits of the form $(j)$ with $j\in\{1,\ldots,n\}$ are written explicitly in the decomposition of $\Gamma(s)$. 
By convention, we omit orbits of the form $(j, \theta)$ in
the decomposition of  $\Gamma(s)$ (this is the reason for writing orbits of length one).  If $\Gamma(s)(j)
=\theta$ for every $1\leq j \leq n$, then we simply denote $\Gamma(s)$ by $\overline{\theta}$.

If the orbit $\varepsilon$ appears in the expression of  $\Gamma(s)$ as
a product of disjoint orbits, then we denote this by $\varepsilon
\subseteq \Gamma(s)$. If $\Gamma(s)(j_{1,1})=j_{1,2},\Gamma(s)(j_{1,2})=j_{1,3}, \ldots, \Gamma(s)(j_{1,p_1-1})=j_{1,p_1},\ldots,\Gamma(s)(j_{q,1})=j_{q,2}, \ldots, \Gamma(s)(j_{q,p_q-1})=j_{q,p_q}$, then we write  $$[j_{1,1},j_{1,2}, \ldots,j_{1,p_1};\ldots;j_{q,1},j_{q,2}, \ldots,j_{q,p_q}] \sqsubseteq \Gamma(s).$$

Note that, if $g\in G$ and $1 \leq n_1, n_2 \leq n$ with
$n_1 \neq n_2$ then
$$\Gamma((g;n_1,n_2)) = (n_1,n_2,\theta) \mbox{  and } \Gamma((g;n_1,n_1)) = (n_1).$$
Therefore, if the group $G$ is trivial, then the elements of $M$ may be viewed as transformations.


Also, for every $s \in S$, we  recall a map
$$
\Psi(s): \{1, \ldots, n\}\cup \{\theta\}  \longrightarrow G\cup \{\theta\}
$$
as follows
$$\Psi(s)(j)=g  \;\; \;\; \mbox{ if } \Gamma(s)(j) \neq \theta
\mbox{ and } (1_G;i,j)s=(g;i,\Gamma(s)(j))$$ for some  $1\leq i \leq
n,$ otherwise $\Psi(s)(j)=\theta$. It is straightforward to verify
that $\Psi$ is well-defined.
%


Let $T$ be a semigroup with a zero $\theta_T$ and let $M$ be a regular Rees matrix semigroup $\mathcal{M}^0(\{1\},n,n;I_n)$.
Let $\Delta$ be a representation of $T$ in the full transformation semigroup on the set $\{ 1, \ldots, n\} \cup \{\theta\}$ such that for
every $t\in T$, $\Delta(t) (\theta ) =\theta$, $\Delta^{-1}(\overline{\theta})= \{\theta_T\}$, and
$\Delta(t)$ restricted to  $\{ 1, \ldots, n\} \setminus \Delta(t)^{-1}(\theta)$ is injective. 
The semigroup $S=M \cup^{\Delta} T$ is the $\theta$-disjoint union of $M$ and $T$ (that is the disjoint union with the zeros identified).
The multiplication is such that $T$ and $M$ are subsemigroups, 
$$(1;i,j) \, t = \left\{ \begin{array}{ll}
                              (1; i,\Delta (t)(j)) & \mbox{ if } \Delta (t)(j) \neq \theta\\
                              \theta & \mbox{ otherwise,}
                          \end{array} \right.
                          $$
and
$$ t(1;i,j) = \left\{ \begin{array}{ll}
    (1;i',j) & \mbox{  if }  \Delta (t)(i')=i\\
    \theta &\mbox{ otherwise. }
    \end{array} \right.
    $$
For more details see \cite{Jes-Sha3}.

Let $\mathsf{V}$ be a pseudovariety of finite semigroups.
A pro-$\mathsf{V}$ semigroup is a compact semigroup that is residually
in $\mathsf{V}$. In case $\mathsf{V}$ consists of all finite semigroups, we call pro-$\mathsf{V}$ semigroups profinite semigroups.
We denote by $\overline{\Omega}_{A}\mathsf{V}$ the free pro-$\mathsf{V}$ semigroup on the set $A$ and by $\Omega_{A}\mathsf{V}$ the free
semigroup in the (Birkhoff) variety generated by $\mathsf{V}$. Such free objects are characterized
by appropriate universal properties. For instance, $\overline{\Omega}_{A}\mathsf{V}$ comes endowed with a
mapping $\iota\colon A \rightarrow \overline{\Omega}_{A}\mathsf{V}$ such that, for every mapping $\phi\colon A \rightarrow S$ into a pro-$\mathsf{V}$ semigroup $S$, there exists a unique continuous homomorphism $\widehat{\phi}\colon \overline{\Omega}_{A}\mathsf{V} \rightarrow S$ such that $\widehat{\phi}\circ\iota=\phi$. For more details on this topic we refer the reader to \cite{Alm}. 

Let $S$ be a finite semigroup. Let $\pi_1,\ldots,\pi_r\in  \overline{\Omega}_{r}\mathsf{V}$.
Define recursively a sequence
$(u_{1,i},\ldots,u_{r,i})$ by $(u_{1,0},\ldots,u_{r,0})\in S^r$ and $u_{i,n+1} = \pi_i(u_{1,n},\ldots,u_{r,n})$. Denote $\lim_{n \rightarrow \infty}u_{i,n!}$ by $\circ_i^{\omega}(\pi_1,\ldots,\pi_r)$.
The component $\circ_i^{\omega}(\pi_1,\ldots,\pi_r)$ for $1\leq i \leq r$ is also a member of $\overline{\Omega}_{r}\mathsf{V}$. Moreover, if each $\pi_i$  is a computable operation, then so is each $\circ_i^{\omega}(\pi_1,\ldots,\pi_r)$ \cite[Corollary 2.5]{Alm4}.
Recall that a pseudoidentity (over $\mathsf{V}$) is a formal equality $\pi = \rho$ between  $\pi,\rho\in  \overline{\Omega}_{r}\mathsf{V}$ for some integer $r$.
For a set $\Sigma$ of $\mathsf{V}$-pseudoidentities, we denote by $\llbracket \Sigma \rrbracket_{\mathsf{V}}$ (or simply $\llbracket \Sigma \rrbracket$ if $\mathsf{V}$ is understood from the context) the class of all $S \in \mathsf{V}$ that satisfy all pseudoidentities from $\Sigma$. Reiterman \cite{Rei} proved that a subclass $\mathsf{V}$ of a pseudovariety $\mathsf{W}$ is a pseudovariety if and only if $\mathsf{V}$ is of the form $\llbracket \Sigma \rrbracket_{\mathsf{W}}$ for some set $\Sigma$ of $\mathsf{W}$-pseudoidentities. For the pseudovarieties $\mathsf{G_{nil}}$
and $\mathsf{BG}$, of all finite block groups, that is, finite semigroups in which each element has at most one inverse, we have 
\begin{center}
$\mathsf{G_{nil}}=\llbracket \phi^{\omega}(x)=x^{\omega},x^{\omega}y=yx^{\omega}=y  \rrbracket$ 
\end{center}
where $\phi$ is the continuous endomorphism of the free profinite semigroup on $\{x,y\}$ such that
$\phi(x)=x^{\omega-1}y^{\omega-1}xy,\phi(y)=y$ \cite[Example 4.15(2)]{Alm2}
and 
\begin{center}
$\mathsf{BG}=\llbracket (ef)^{\omega}=(fe)^{\omega} \rrbracket$ 
\end{center}
where $e=x^{\omega},f=y^{\omega}$ (see for example \cite[Exercise 5.2.7]{Alm}).



\section{Strongly nilpotent semigroups}
For a semigroup $S$ with elements $x_1,\ldots,x_t, z_{1},z_{2}, \ldots$ one recursively defines sequences 
$$\lambda_{n,i}=\lambda_{n,i}(x_1,\ldots,x_t;z_{1},\ldots, z_{n})$$ by
$\lambda_{0,i}=x_i$ and
$$\lambda_{n+1,i}=\lambda_{n,i} z_{n+1} \lambda_{n,i+1}z_{n+1}\,\cdots\, \lambda_{n,t}z_{n+1}\lambda_{n,1} z_{n+1}\,\cdots\, \lambda_{n,i-1}$$ 
for every $1\leq i\leq t$. 
A semigroup is said to be \emph{strongly Mal'cev nilpotent}, if there exists a positive integer $n$ such that 
$$\lambda_{n,1}(x_1,\ldots,x_t;z_{1},\ldots, z_{n})=\cdots=\lambda_{n,t}(x_1,\ldots,x_t;z_{1},\ldots, z_{n})$$ for all $x_1,\ldots,x_t$ in $S$ and $z_{1}, \ldots , z_{n}$ in $S^{1}$. The smallest such $n$ is called the \emph{strong Mal'cev nilpotency class} of $S$. 
We denote the class all finite strongly Mal'cev nilpotent semigroups by $\mathsf{SMN}$.
Note that if we choose $t=2$ then the sequences $\lambda_{n,1}$ and $\lambda_{n,2}$ are equal to the sequences $\lambda_{n}$ and $\rho_{n}$, respectively.
Hence, if a semigroup $S$ is strongly Mal'cev nilpotent then it is Mal'cev nilpotent too and, thus, we have $\mathsf{SMN}\subseteq \mathsf{MN}$. The set $\mathsf{SMN}$ is a pseudovariety. It is an example of ultimate equational definition of pseudovariety in the sense of Eilenberg and Sch{\"u}tzenberger \cite{Eil-Sch}. 
Since $\mathsf{SMN}\subseteq \mathsf{MN}$ and $\mathsf{MN}\subsetneqq \mathsf{BG}_{nil}$, we have the following theorem.

\begin{thm}\label{MN'BGnil}
We have $\mathsf{SMN}\subsetneqq \mathsf{BG}_{nil}$.
\end{thm}

Note that we can improve the definition of Mal'cev nilpotency for finite semigroups.
\begin{lem} \label{finite-nil-dif}
A finite semigroup $S$ is Mal'cev nilpotent if and only if there exists a positive integer $n$ such that 
$$\lambda_{n}(x,y,z_{1},\ldots, z_{n}) = \rho_{n}(x,y,z_{1},\ldots, z_{n})$$ for all $x,y,z_{1}, \ldots, z_{n}$
in $S$.
\end{lem}

\begin{proof}
Suppose that there exists a finite semigroup $S$ such that $S$ satisfies the condition of the lemma and $S$ is not Mal'cev nilpotent. 

If $S\not\in \mathsf{BG}_{nil}$, then there exists a regular ${\mathcal J}$-class $\mathcal{M}^0(G,n,m;P)\setminus \{\theta\}$ of $S$ such that one of the following conditions holds:
\begin{enumerate}
\item $G$ is not a nilpotent group;
\item there exist integers $1\leq i_1,i_2\leq n$ and $1\leq j\leq m$ such that $p_{ji_1},p_{ji_2}\neq \theta$;
\item there exist integers $1\leq i\leq n$ and $1\leq j_1,j_2\leq m$ such that $p_{j_1i},p_{j_2i}\neq \theta$.
\end{enumerate}
If $G$ is not a nilpotent group, then, by \cite[Corollary 1]{Neu-Tay}, $G$ is not Mal'cev nilpotent. 
Since $G$ has an identity, $G$ does not satisfy the condition of the lemma, a contradiction.
If (2) holds, then
\begin{align*}
&\lambda_{n}((1_G;i_1,j),(1_G;i_2,j),(1_G;i_1,j),(1_G;i_1,j),\ldots, (1_G;i_1,j)) \neq\\
&\rho_{n}((1_G;i_1,j),(1_G;i_2,j),(1_G;i_1,j),(1_G;i_1,j),\ldots, (1_G;i_1,j)),
\end{align*} for every integer $0\leq n$. A contradiction with the assumption. Similarly, we have a contradiction for Condition (3).

Now, suppose that $S\in \mathsf{BG}_{nil}$. Since $S$ satisfies the condition of the lemma and $S\not\in\mathsf{MN}$, by Lemma~\ref{finite-nilpotent} there exist a positive integer $m$, distinct elements $x, y\in S$ and elements 
$ w_{1}, w_{2}, \ldots, w_{m-1}\in S^{1}$ such that 
\begin{eqnarray} \label{nil-fi}
~x = \lambda_{m}(x, y, 1,w_{1}, \ldots, w_{m-1}) \mbox{ and } y = \rho_{m}(x,y, 1,w_{1},\ldots, w_{m-1}).
\end{eqnarray}
As $x\neq y$ and $S\in \mathsf{BG}_{nil}$, by \eqref{nil-fi}, there exists a regular ${\mathcal J}$-class $$M=\mathcal{M}^0(G,n,n;I_n)\setminus \{\theta\}$$ of $S$ such that $x,y\in M$. Then, there exist elements $(g;i,j),(g';i',j')\in M$ such that $x=(g;i,j)$ and $y=(g';i',j')$. As $\lambda_1,\rho_1\in M$, we have $j=i'$ and $j'=i$. Hence, we have $(i,j)\subseteq \Gamma(w_1)$ when $\Gamma$ is an ${\mathcal L}_{M}$-representation of $M$. If $i\neq j$ then $w_1\neq 1$ and, thus,
\begin{align*}
&\lambda_{n}((1_G;i,i),(1_G;j,j),w_1,w_1^2,w_1,w_1^2,\ldots) \neq\\
&\rho_{n}((1_G;i,i),(1_G;j,j),w_1,w_1^2,w_1,w_1^2,\ldots),
\end{align*} for every integer $0\leq n$. A contradiction with the assumption. If $i=j$, then we have 
\begin{align*}
&g=\lambda_{n}(g,g',1,\psi(w_1)(i),\psi(w_2)(i),\ldots,\psi(w_{m-1})(i)),\\
&g'=\rho_{n}(g,g',1,\psi(w_1)(i),\psi(w_2)(i),\ldots,\psi(w_{m-1})(i)).
\end{align*} Since $x\neq y$, we have $g\neq g'$. It follows that $G$ is not nilpotent. 

The result follows.
\end{proof}

Now, by Lemma~\ref{finite-nil-dif} with using the same method as in the proof of \cite[Lemma 2.2]{Jes-Sha}, we can improve Lemma~\ref{finite-nilpotent}. 
\begin{lem} \label{finite-nilpotent-2}
A finite semigroup $S$ is not Mal'cev nilpotent if and only if there exist a positive integer $m$, distinct elements $x, y\in S$ and elements 
$ w_{1}, w_{2}, \ldots, w_{m}\in S$ such that $$x = \lambda_{m}(x, y, w_{1}, w_{2}, \ldots, w_{m}) \mbox{ and } y = \rho_{m}(x,y, w_{1}, w_{2}, \ldots, w_{m}).$$
\end{lem}


Neumann and Taylor proved that a group $G$ is nilpotent with the nilpotency class $n$ if and only if it is Mal'cev nilpotent with the nilpotency $n$ \cite[Corollary 1]{Neu-Tay}. The following lemma, presents a similar result for strong Mal'cev nilpotency.

\begin{lem} \label{nil-group}
Let $n \geq 1$. Then the following conditions are equivalent for
a group $G$:
\begin{enumerate}
\item $G$ is a nilpotent group of class $n$.
\item $G$ is strongly Mal'cev nilpotent with the strong Mal'cev nilpotency class $n$.
\end{enumerate}
\end{lem}

\begin{proof}
If $G$ is strongly Mal'cev nilpotent with the strong Mal'cev nilpotency class $1$, then 
$\lambda_{1,1}(x_1,x_2;1)=\lambda_{1,2}(x_1,x_2;1)$, 
for all $x_1,x_2$ in $G$. It follows that $x_1x_2 =x_2x_1$.
Thus (1), (2) are equivalent to the commutativity of $G$. 

Assume that the assertion holds for some $n > 1$. Let $x_1,\ldots,x_t,z_{1}, \ldots , z_{n}$ in $G$ and $a_i=\lambda_{n,i}(x_1,\ldots,x_t;z_{1},\ldots, z_{n})$ for every $1\leq i\leq t$.
For any $x \in G$, denote by $\overline{x}$ the image of
$x$ in $G/Z(G)$. If $G$ is nilpotent of class $n + 1$, then $G / Z (G)$ is nilpotent of class $n$ and, 
by the induction hypothesis we have
$$\lambda_{n,1}(\overline{x}_1,\ldots,\overline{x}_t;\overline{z}_{1},\ldots, \overline{z}_{n})=\cdots=
\lambda_{n,t}(\overline{x}_1,\ldots,\overline{x}_t;\overline{z}_{1},\ldots, \overline{z}_{n}).$$
Thus, there exist elements $v_{i,j}\in Z(G)$ such that 
$$a_i=\lambda_{n,i}(x_1,\ldots,x_t;z_{1},\ldots, z_{n})=\lambda_{n,j}(x_1,\ldots,x_t;z_{1},\ldots, z_{n})v_{i,j}=a_jv_{i,j}$$
for all $1\leq i,j\leq t$ with $i\neq j$.
Let $1\leq k,i,j\leq t$ with $i\neq j$ and let $$(b_1,\ldots,b_t)=(a_{k},\ldots,a_t,a_1,\ldots,a_{k-1}).$$ There exist integers $g$ and $h$ such that $b_g=a_i$ and $b_h=a_j$.
Since $a_i=a_jv_{i,j}$ and $v_{i,j}\in Z(G)$, we have 
\begin{align*} 
\lambda_{n+1,k}(x_1,\ldots,x_t;z_{1},\ldots, z_{n+1})
&=b_1z_{n+1}b_2z_{n+1}\ldots b_g\ldots b_h\ldots z_{n+1}b_t\\
&=b_1z_{n+1}b_2z_{n+1}\ldots a_i\ldots a_j\ldots z_{n+1}b_t\\
&=b_1z_{n+1}b_2z_{n+1}\ldots a_jv_{i,j}\ldots a_j\ldots z_{n+1}b_t\\
&=b_1z_{n+1}b_2z_{n+1}\ldots a_j\ldots a_jv_{i,j}\ldots z_{n+1}b_t\\
&=b_1z_{n+1}b_2z_{n+1}\ldots a_j\ldots a_i \ldots z_{n+1}b_t.
\end{align*}
Since we take the elements $a_i$ and $a_j$ arbitrarily, we have $$\lambda_{n+1,k}(x_1,\ldots,x_t;z_{1},\ldots, z_{n+1})=a_1z_{n+1}a_2z_{n+1}\ldots  z_{n+1}a_t.$$
Therefore $G$ is strongly Mal'cev nilpotent with the strong Mal'cev nilpotency class $n+1$.

Now, assume that $G$ is strongly Mal'cev nilpotent with the strong Mal'cev nilpotency class $n+1$. Hence $G$ is Mal'cev nilpotent with the nilpotency class $n'$ with $n'\leq n+1$. Then, by \cite[Corollary 1]{Neu-Tay}, $G$ is a nilpotent group with the nilpotency class $n'$. If $n'<n+1$, then, by assertion, $G$ is strongly Mal'cev nilpotent with the strong Mal'cev nilpotency class $n'$, a contradiction. Hence, $G$ is a nilpotent group with the nilpotency $n+1$.
\end{proof}

As was mentioned about Lemma \ref{finite-nilpotent}, it is proved in \cite{Jes-Sha} that a finite semigroup $S$ is not Mal'cev nilpotent if and only if there exist a positive integer $m$, distinct elements $x, y\in S$ and elements 
$ w_{1}, w_{2}, \ldots, w_{m}\in S^{1}$ such that $x = \lambda_{m}(x, y, w_{1}, w_{2}, \ldots, w_{m})$ and $y = \rho_{m}(x,y, w_{1}, w_{2}, \ldots, w_{m})$. We proceed with some lemmas that serve to give a criterion for finite semigroups not to be strongly Mal'cev nilpotent (Lemma~\ref{finite-nilpotent'}). 

\begin{lem} \label{[,]1}
Let $S$ be a finite semigroup.
Suppose that 
$$S= S_1 \supset S_2 \supset \ldots \supset S_{s} \supset S_{s+1} = \emptyset$$
is a principal series of $S$ and there is an integer $1\leq p\leq s$ such that the following conditions are satisfied:
\begin{enumerate}
\item $S_p / S_{p+1}$ is an inverse completely $0$-simple semigroup, say $$M=\mathcal{M}^{0}(G,q,q;I_q);$$
\item there exist an integer $1<t$, integers $\alpha_i,\beta_i$ $(1\leq i\leq t)$, and elements $v_1,\ldots,v_t\in S\setminus S_{p+1}$ with $$[\beta_1,\alpha_{1+i\pmod t};\ldots;\beta_t,\alpha_{t+i\pmod t}]\sqsubseteq \Gamma(v_i)\quad (1\leq i\leq t),$$  where $\Gamma$ is an ${\mathcal L}_M$-representation of $S/S_{p+1}$;
\item $1<\abs{\{\alpha_1,\ldots,\alpha_t\}}<t$ or $1<\abs{\{\beta_1,\ldots,\beta_t\}}<t$.
\end{enumerate}
Then, there exists an integer $t'$ such that the following conditions are satisfied:
\begin{enumerate}
\item $t'\neq 1$ and $t'\mid t$;
\item $\abs{\{\alpha_1,\ldots,\alpha_{t'}\}}=t'$;
\item $[\beta_1,\alpha_{1+i\pmod {t'}};\ldots;\beta_{t'},\alpha_{t'+i\pmod {t'}}]\sqsubseteq \Gamma(v_i)$ $(1\leq i\leq t')$.
\end{enumerate}
\end{lem}

\begin{proof}
We have $1<\abs{\{\alpha_1,\ldots,\alpha_t\}}<t$ or $1<\abs{\{\beta_1,\ldots,\beta_t\}}<t$.
First, we assume that $1<\abs{\{\alpha_1,\ldots,\alpha_t\}}<t$. 
Since $$\abs{\{\alpha_1,\ldots,\alpha_t\}}<t,$$
there exist integers $1\leq h_1<h_2\leq t$ such that $\alpha_{h_1}=\alpha_{h_2}$ and if $\abs{h_4-h_3}<h_2-h_1$, for some distinct integers $h_3$ and $h_4$, then $\alpha_{h_3}\neq\alpha_{h_4}$. 
First, suppose that $h_2-h_1=1$. Since 
$$[\beta_1,\alpha_{1};\ldots;\beta_t,\alpha_{t}]\sqsubseteq \Gamma(v_t) \mbox{ and } [\beta_1,\alpha_{2};\beta_2,\alpha_{3};\ldots;\beta_t,\alpha_{1}]\sqsubseteq \Gamma(v_1),$$
we have $\alpha_1=\cdots=\alpha_t$, which contradicts the initial assumption. Now, suppose that $1<h_2-h_1$.
By our assumption, the integers $\alpha_1,\ldots,\alpha_{(h_2-h_1)}$ are pairwise distinct. Again, as 
$$[\beta_1,\alpha_{1};\ldots;\beta_t,\alpha_{t}]\sqsubseteq \Gamma(v_t) \mbox{ and } [\beta_1,\alpha_{2};\beta_2,\alpha_{3};\ldots;\beta_t,\alpha_{1}]\sqsubseteq \Gamma(v_1),$$ we have $\alpha_j=\alpha_{j+\gamma(h_2-h_1)\pmod t}$, for every $0\leq \gamma$ and $1\leq j\leq (h_2-h_1)$. Also, since the integers $\alpha_1,\ldots,\alpha_{(h_2-h_1)}$ are pairwise distinct, we have $(h_2-h_1)\mid t$. Therefore, we have
$$[\beta_1,\alpha_{1+i\pmod{(h_2-h_1)}};\ldots;\beta_{(h_2-h_1)},\alpha_{(h_2-h_1)+i\pmod{(h_2-h_1)}}]\sqsubseteq \Gamma(v_i),$$ for every $1\leq i\leq h_2-h_1$.

The proof in case $1<\abs{\{\beta_1,\ldots,\beta_t\}}<t$ is similar.
\end{proof}

We can get the following lemma from the results of the paper \cite{Jes-Sha3}. We present a similar lemma as well as the analogous result for strong Mal'cev nilpotency (Lemma~\ref{[,]}).
 
\begin{lem} \label{[,]'}
Let $S\in \mathsf{BG}_{nil}$. The semigroup $S$ is not Mal'cev nilpotent if and only if there exist ideals $A,B$ of $S$, an inverse Rees matrix semigroup $M=\mathcal{M}^{0}(G,q,q;I_q)$, and elements $x,y,w,v$ such that the following conditions are satisfied:
\begin{enumerate}
\item $B\subsetneqq A$ and $A/B\cong M$;
\item $x=(g;\alpha,\beta),y=(g';\alpha',\beta')\in M$ and $\alpha\neq\alpha'$;
\item $w,v\in S\setminus B$, $[\beta,\alpha';\beta',\alpha] \sqsubseteq \Gamma(w)$ and $[\beta',\alpha';\beta,\alpha] \sqsubseteq \Gamma(v)$, where $\Gamma$ is an ${\mathcal L}_M$-representation of $S/B$.
\end{enumerate}
\end{lem}

\begin{lem} \label{[,]}
Let $S\in \mathsf{BG}_{nil}$. The semigroup $S$ is not strongly Mal'cev nilpotent if and only if there exist ideals $A,B$ of $S$, an inverse Rees matrix semigroup $M=\mathcal{M}^{0}(G,q,q;I_q)$, an integer $t$ and elements $y_1,\ldots,y_t,v_1,\ldots,v_t$ such that the following conditions are satisfied:
\begin{enumerate}
\item $B\subsetneqq A$ and $A/B\cong M$;
\item $1<t$;
\item $y_i=(g_i;\alpha_i,\beta_i)\in M$ $(1\leq i\leq t)$ and $\abs{\{\alpha_1,\ldots,\alpha_t\}}=t$;
\item $v_1,\ldots,v_t\in S\setminus B$ and $[\beta_1,\alpha_{1+i\pmod t};\ldots;\beta_t,\alpha_{t+i\pmod t}]\sqsubseteq \Gamma(v_i)$ $(1\leq i\leq t)$, where $\Gamma$ is an ${\mathcal L}_M$-representation of $S/B$.
\end{enumerate}
\end{lem}

\begin{proof}
First, suppose that $S$ is not strongly Mal'cev nilpotent.  
Let $k=\abs{S}$. Since $S$ is not strongly Mal'cev nilpotent, there exist elements
$a_1,\ldots,a_t\in S$ with $t>1$, and $w_{1}, \ldots, w_{k^t+1}$ $ \in S^1$ such
that
$$
\abs{\{\lambda_{k^t+1,1}(a_1,\ldots,a_t;w_{1},\ldots, w_{k^t+1}),\ldots,\lambda_{k^t+1,t}(a_1,\ldots,a_t;w_{1},\ldots, w_{k^t+1})\}}\neq 1
.$$

Since
$\abs{S^t} =k^t$, there exist positive integers $r_1$ and $r_2
\leq k^t+1$ with $r_1< r_2$ such that
\begin{align*}
&(\lambda_{r_1,1}(a_1,\ldots,a_t;w_{1},\ldots, w_{r_1}),\ldots,\lambda_{r_1,t}(a_1,\ldots,a_t;w_{1},\ldots, w_{r_1}))\\ 
&=(\lambda_{r_2,1}(a_1,\ldots,a_t;w_{1},\ldots, w_{r_2}),\ldots,\lambda_{r_2,t}(a_1,\ldots,a_t;w_{1},\ldots, w_{r_2})).
\end{align*} 
Put $y_i= \lambda_{r_1,i}(a_1,\ldots,a_t;w_{1},\ldots, w_{r_1})$ $(1\leq i\leq t)$, $m=r_2-r_1$, and $v_j=w_{r_1+j}$ $(1\leq j\leq m)$. This gives the equalities
\begin{eqnarray}\label{lambda-kt}
&y_i= \lambda_{m,i}(y_1,\ldots,y_t;v_{1},\ldots, v_{m})\quad (1\leq i\leq t).
\end{eqnarray}
Since $$\abs{\{\lambda_{k^t+1,1}(a_1,\ldots,a_t;w_{1},\ldots, w_{k^t+1}),\ldots,\lambda_{k^t+1,t}(a_1,\ldots,a_t;w_{1},\ldots, w_{k^t+1})\}}\neq 1,$$ we have $1<\abs{\{y_1,\ldots,y_t\}}$. 
Let 
$$S= S_1 \supset S_2 \supset \ldots \supset S_{s} \supset S_{s+1} = \emptyset$$
be a principal series of $S$. Suppose that $y_1 \in S_{p} \setminus
S_{p+1}$ for some $1\leq p\leq s$. Because $S_p$ and $S_{p+1}$ are
ideals of $S$, the equalities \eqref{lambda-kt} yield $y_1,\ldots,y_t \in S_p \setminus
S_{p+1}$ and $v_{1},\ldots, v_{m} \in S \setminus S_{p+1}$.
Since $S\in \mathsf{BG}_{nil}$, $S_p / S_{p+1}$ is an inverse completely
$0$-simple semigroup, say $M=\mathcal{M}^{0}(G,q,q;I_q)$. Then there exist integers $1\leq \alpha_{i},\beta_{i}\leq q$ and elements $g_{i}\in G$ such that $y_i=(g_{i};\alpha_{i},\beta_{i})$ $(1\leq i\leq t)$.
The equalities \eqref{lambda-kt}, imply that
$$[\beta_1,\alpha_{1+i\pmod t};\ldots;\beta_t,\alpha_{t+i\pmod t}]\sqsubseteq \Gamma(v_i)\quad (1\leq i\leq t).$$
If $\alpha_1=\cdots=\alpha_t=\alpha$ and $\beta_1=\cdots=\beta_t=\beta$, then $[\beta,\alpha]\sqsubseteq \Gamma(v_i)$ $(1\leq i\leq m)$. Therefore, we have 
$$g_i= \lambda_{m,i}(g_1,\ldots,g_t;\Psi(v_{1})(\beta),\ldots, \Psi(v_{m})(\beta))\quad (1\leq i\leq t).$$
Since $1<\abs{\{y_1,\ldots,y_t\}}$, we have $1<\abs{\{g_1,\ldots,g_t\}}$. Then, by Lemma~\ref{nil-group}, $G$ is not a nilpotent group. This contradicts the assumption that $S\in \mathsf{BG}_{nil}$. Then, there exist distinct integers $1\leq h,h'\leq t$ such that $\alpha_h\neq \alpha_{h'}$ or $\beta_h\neq \beta_{h'}$. Now, by Lemma~\ref{[,]1}, 
there exists an integer $t'$ such that the following conditions are satisfied:
\begin{enumerate}
\item $t'\neq 1$ and $t'\mid t$;
\item $\abs{\{\alpha_1,\ldots,\alpha_{t'}\}}=t'$;
\item $[\beta_1,\alpha_{1+i\pmod {t'}};\ldots;\beta_{t'},\alpha_{t'+i\pmod {t'}}]\sqsubseteq \Gamma(v_i)$ $(1\leq i\leq t')$.
\end{enumerate}

The converse, follows at once from the definition of strong Mal'cev nilpotency.  
\end{proof}

Now, we can improve the definition of strong Mal'cev nilpotency for finite semigroups as well as Lemma~\ref{finite-nil-dif}.
\begin{lem} \label{finite-nilpotent-2'}
A finite semigroup $S$ is strongly Mal'cev nilpotent if and only if there exists a positive integer $n$ such that 
$$\lambda_{n,1}(x_1,\ldots,x_t;z_{1},\ldots, z_{n})=\cdots=\lambda_{n,t}(x_1,\ldots,x_t;z_{1},\ldots, z_{n})$$ for all $x_1,\ldots,x_t,z_{1}, \ldots , z_{n}$ in $S$.
\end{lem}

\begin{proof}
Suppose that there exists a finite semigroup $S$ such that $S$ satisfies the condition of the lemma and $S$ is not strongly Mal'cev nilpotent. 

If $S\not\in \mathsf{BG}_{nil}$, then $S\not\in\mathsf{MN}$ and, thus, by Lemma~\ref{finite-nil-dif}, $S$ does not satisfy the condition of the lemma, a contradiction.

Now, suppose that $S\in \mathsf{BG}_{nil}$. Since $S\not\in\mathsf{SMN}$ and $S$ satisfies the condition of the lemma, by Lemma~\ref{[,]}, there exist 
a regular ${\mathcal J}$-class $M=\mathcal{M}^0(G,n,n;I_n)\setminus \{\theta\}$ of $S$, a positive integer $t>1$, elements 
$y_i=(g_i;\alpha_i,\beta_i)\in M$, for every $1\leq i\leq t$
and an element 
$w\in S^{1}$
such that $\abs{\{\alpha_1,\ldots,\alpha_t\}}=t$ and
$\lambda_{2,i}(y_1,\ldots,y_t;1,w)\in M$ ($1\leq i\leq t$). 
As $\lambda_{1,1},\ldots,\lambda_{1,t}\in M$, we have $\beta_i=\alpha_{i+1}$ $(1\leq i\leq t-1)$ and $\beta_t=\alpha_1$. 
Hence, we have $(\alpha_1,\ldots,\alpha_t)\subseteq \Gamma(w)$ when $\Gamma$ is an ${\mathcal L}_{M}$-representation of $M$. Since 
$\abs{\{\alpha_1,\ldots,\alpha_t\}}=t$ and $t>1$, we have $w\neq 1$.
Now, as $(\alpha_1,\ldots,\alpha_t)\subseteq \Gamma(w)$, we have
\begin{align*}
&\lambda_{l,i}((1_G;\alpha_1,\alpha_1),\ldots,(1_G;\alpha_t,\alpha_t);w,w^2,\ldots,w^t,w,\ldots,w^t,\ldots)=\\
&(k_l;\alpha_i,\alpha_{i-l\mod t})\quad (1\leq i\leq t,0\leq l),
\end{align*}
for some element $k_l\in G$ and, thus,
\begin{align*}
&\lambda_{l,i}((1_G;\alpha_1,\alpha_1),\ldots,(1_G;\alpha_t,\alpha_t);w,w^2,\ldots,w^t,w,\ldots,w^t,\ldots) \neq\\
&\lambda_{l,i'}((1_G;\alpha_1,\alpha_1),\ldots,(1_G;\alpha_t,\alpha_t);w,w^2,\ldots,w^t,w,\ldots,w^t,\ldots),
\end{align*} for all integers $1\leq i<i'\leq t$ and $0\leq l$. Since $w\neq 1$, there is a contradiction with the assumption. 
\end{proof}

\begin{lem} \label{finite-nilpotent'}
A finite semigroup $S$ is not strongly Mal'cev nilpotent if and only if there exist positive integers $t>1,m$, pairwise distinct elements $a_1,\ldots,a_t$ in $S$ and elements 
$w_{1}, w_{2}, \ldots, w_{m}$ in $S$ such that $$a_i = \lambda_{m,i}(a_1,\ldots,a_t;w_{1},\ldots, w_{m}),$$ for all $1\leq i\leq t$.
\end{lem}

\begin{proof}
If $S\in \mathsf{BG}_{nil}$ and $S\not\in \mathsf{SMN}$, as we argue in the proof of Lemma~\ref{finite-nilpotent-2'},
by Lemma~\ref{[,]}, the result follows.
Now, suppose that $S\not\in \mathsf{BG}_{nil}$. Hence, we have $S\not\in \mathsf{MN}$ and the result follows from Lemma~\ref{finite-nil-dif}.
\end{proof}

The Rees matrix semigroup $M=\mathcal{M}^{0}(G, n,m;P)$ is Mal'cev nilpotent if and only if $G$ is nilpotent and $M$ is inverse
\cite[Lemma 2.1]{Jes-Okn}. Now, by Lemma~\ref{finite-nilpotent'}, we can present a similar result for the strong Mal'cev nilpotency.

\begin{lem} \label{Rees'}
The finite Rees matrix semigroup $M=\mathcal{M}^{0}(G, n,m;P)$ is strongly Mal'cev nilpotent if and only if $G$ is nilpotent and $M$ is inverse.
\end{lem}

\begin{proof}
If $M$ is strongly Mal'cev nilpotent then $M$ is Mal'cev nilpotent and, thus, by \cite[Lemma 2.1]{Jes-Okn}, the result follows. Now, suppose that $G$ is nilpotent, $M$ is inverse, and $S$ is not strongly Mal'cev nilpotent. Then, by Lemma~\ref{finite-nilpotent'}, there exist positive integers $t>1,m$, pairwise distinct elements $a_i=(g_i;\alpha_i,\beta_i)\in M$ ($1\leq i\leq t$) and elements 
$w_{1}, \ldots, w_{m}\in M$ such that $$a_i = \lambda_{m,i}(a_1,\ldots,a_t;w_{1},\ldots, w_{m})\quad (1\leq i\leq t).$$  Since $M$ is inverse and
$a_i = \lambda_{m,i}(a_1,\ldots,a_t;w_{1},\ldots, w_{m})$ ($1\leq i\leq t$), we have
$$[\beta_1,\alpha_{1+j\pmod t};\ldots;\beta_t,\alpha_{t+j\pmod t}]\sqsubseteq \Gamma(w_j)\quad (1\leq j\leq m)$$  where $\Gamma$ is an ${\mathcal L}_M$-representation of $M$. Now, as $w_{1},\ldots, w_{m}\in M$, we have
$$\beta_1=\cdots=\beta_t\mbox{ and }\alpha_1=\cdots=\alpha_t.$$
Therefore, we have 
$$g_i= \lambda_{m,i}(g_1,\ldots,g_t;\Psi(w_{1})(\beta_1),\ldots, \Psi(w_{m})(\beta_1))\quad (1\leq i \leq t).$$ Then $G$ is not nilpotent by Lemma~\ref{nil-group} in contradiction with the initial assumption.
\end{proof}

\section{Sch\"utzenberger graphs}\label{Schutz}

Let $M$ be a finite $A$-generated semigroup.
If $X$ is an $\mathcal{R}$-class of $M$, then the Sch\"utzenberger graph (with respect to $A$) of $X$, denoted $Sch_A(X)$, is the full subgraph of the  right Cayley graph of $M$ with set of vertices $X$. 
Dually, for an $\mathcal{L}$-class $Y$, the left Sch\"utzenberger graph of $Y$, denoted $Sch_A^{\rho}(Y)$, is the full subgraph of the left Cayley graph of $M$ with vertices $Y$. 
An $A$-graph $\Gamma$ is inverse, if and only if for every $w \in (A\cup A^{-1})^{\star}$ there is at run labeled $w$ from any vertex $q$ in the graph $\Gamma$ (for more detail see \cite{Ste01}). It is clear that if $M$ is a finite Mal'cev nilpotent semigroup, then $Sch_A(X)$ is inverse, for every regular $\mathcal{R}$-class $X$, and $Sch_A^{\rho}(Y)$ is inverse, for every regular $\mathcal{L}$-class $Y$. 

Let $X$ be an $\mathcal{R}$-class of $M$ and let $$L_{\beta,\alpha,X}=\{w\in A^+\mid w\mbox{ run from } \beta \mbox{ to } \alpha\},$$
for every $\alpha,\beta\in V(Sch_A(X))$. 
We define the following notions for the $\mathcal{R}$-class $X$:
\begin{enumerate}
\item $X$ is ($\mathcal{H}$-nilpotent) nilpotent in $M$, if there exist vertices $\alpha,\alpha',\beta,\beta'$ in $V(Sch_A(X))$ such that $\alpha\neq\alpha'$ ($\alpha,\alpha'$ are not in the same $\mathcal{H}$-class) and $L_{\beta,\alpha,X}\cap L_{\beta',\alpha',X}\neq\emptyset$, then $L_{\beta',\alpha,X}\cap L_{\beta,\alpha',X}=\emptyset$.
\item $X$ is ($\mathcal{H}$-strongly nilpotent) strongly nilpotent in $M$, if there exist vertices $\alpha_1,\ldots,\alpha_n,\beta_1,\ldots,\beta_n$ in $V(Sch_A(X))$ such that there exist integers $1\leq i,j\leq n$ with $\alpha_i\neq\alpha_j$ ($\alpha_i,\alpha_j$ are in distinct $\mathcal{H}$-classes) and $L_{\beta_1,\alpha_1,X}\cap\cdots \cap L_{\beta_n,\alpha_n,X}\neq\emptyset,$ then there exists an integer $k$ such that  $L_{\beta_1,\alpha_{1+k\pmod n},X}\cap\cdots \cap L_{\beta_n,\alpha_{n+k\pmod n},X}=\emptyset$.
\end{enumerate} 

The following proposition can be seen as a criterion to detect non Mal'cev nilpotent semigroups 
by the Sch\"utzenberger graphs of its regular $\mathcal{R}$-classes.

\begin{prop}\label{Alg1}
Let $S$ be an $A$-generated finite semigroup with the following conditions:
\begin{enumerate}
\item the semigroup $S$ is in the pseudovariety $\mathsf{BG}_{nil}$;
\item if $G$ is a subgroup of $S$ and $g\in G$ with $g\neq 1_G$, then $g^2\neq 1_G$. 
\end{enumerate}
If there exists a regular $\mathcal{R}$-class $X$ of $S$ such that the subset $X$ is not nilpotent in $S$, then the semigroup $S$ is not Mal'cev nilpotent. 
\end{prop}

\begin{proof}
Suppose that there exist a regular $\mathcal{R}$-class $X$ of $S$ and vertices $\alpha,\alpha',\beta,$ and $\beta'$ in $V(Sch_A(X))$ such that $\alpha\neq\alpha'$ and $L_{\beta,\alpha,X}\cap L_{\beta',\alpha',X},L_{\beta',\alpha,X}\cap L_{\beta,\alpha',X}\neq\emptyset$. Let $J$ be a $\mathcal{J}$-class of $X$. There exist an integer $n$ and a finite nilpotent group $G$ such that $J\cup\{\theta\}\cong (M=)\mathcal{M}^{0}(G,n,n;I_n)$. Then, we have $$\alpha=(g_1;a,b_1), \alpha'=(g_2;a,b_2),\beta=(g_3;a,b_3),\mbox{ and }\beta'=(g_4;a,b_4)$$ for some integers $1\leq a,b_1,b_2,b_3,b_4\leq n$ and elements $g_1,g_2,g_3,g_4\in G$. Therefore, there exist elements $m_1,m_2\in S$ such that $[b_3,b_1;b_4,b_2]\sqsubseteq \Gamma(m_1)$ and $[b_4,b_1;b_3,b_2]\sqsubseteq \Gamma(m_2)$ when $\Gamma$ is an ${\mathcal L}_{M}$-representation of $J$. If $b_1\neq b_2$ then, by Lemma~\ref{[,]'}, the semigroup $S$ is not Mal'cev nilpotent. Suppose that $b_1=b_2$. Hence, $g_1\neq g_2$.
There exist elements $g_{m_1},g_{m_2}$ such that $g_3g_{m_1}=g_1$, $g_4g_{m_1}=g_2$, $g_4g_{m_2}=g_1$ and $g_3g_{m_2}=g_2$. 
%
%
It follows that $g_3g_4^{-1}=g_1g_2^{-1}$ and $g_3g_4^{-1}=g_2g_1^{-1}$ and, thus, $(g_2g_1^{-1})^2=1_G$. A contradiction with the assumption.
\end{proof}

The following propositions can be seen as criteria to detect non Mal'cev nilpotent semigroups and non strongly Mal'cev nilpotent semigroups that are obtained at once from Lemmas~\ref{[,]'}, \ref{[,]1} and \ref{[,]}.

\begin{prop}\label{Alg3}
Let $S$ be an $A$-generated semigroup in the pseudovariety $\mathsf{BG}_{nil}$. The following conditions hold:
\begin{enumerate}
\item if there exists a regular $\mathcal{R}$-class $X$ of $S$ such that the subset $X$ is not $\mathcal{H}$-nilpotent in $S$, then the semigroup $S$ is not Mal'cev nilpotent. 
\item if there exists a regular $\mathcal{R}$-class $X$ of $S$ such that the subset $X$ is not $\mathcal{H}$-strongly nilpotent in $S$, then the semigroup $S$ is not strongly Mal'cev nilpotent.
\end{enumerate}
\end{prop}

We recall the pseudovariety $$\mathsf{BI} = \{S\in \mathsf{S}\mid S\mbox{~is block group and all subgroups of $S$ are trivial}\}$$ where $\mathsf{S}$ is all finite semigroups.

\begin{prop}\label{Alg4}
Let $S$ be an $A$-generated semigroup in the pseudovariety $\mathsf{BI}$. The following conditions hold:
\begin{enumerate}
\item the semigroup $S$ is Mal'cev nilpotent if and only if for every regular $\mathcal{R}$-class $X$ of $S$ the subset $X$ is nilpotent in $S$. 
\item the semigroup $S$ is strongly Mal'cev nilpotent if and only if for every regular $\mathcal{R}$-class $X$ of $S$ the subset $X$ is strongly nilpotent in $S$. 
\end{enumerate}
\end{prop}
\section{An iterative description of $\mathsf{SMN}$}\label{iterSMN}

Let $$\mathsf{SMN^{\circ}_t}=\llbracket
\phi_t^{\omega}(y_1)=\cdots =\phi_t^{\omega}(y_t)
\rrbracket$$
where $\phi_t$ is the continuous endomorphism of the free profinite
semigroup on $\{y_1,\ldots,y_t,z_1,\ldots,z_t\}$ such that $\phi_t(y_i)=\lambda_{t,i}(y_1,\ldots,y_t;z_1,\ldots,z_t)$ and $\phi_t(z_i)=z_i$, for all $1\leq i\leq t.$

\begin{thm}\label{MN-omega5}
We have $\mathsf{SMN}=(\bigcap_{2\leq t}\mathsf{SMN^{\circ}_t})$.
\end{thm}

\begin{proof}
First, we prove that $\mathsf{SMN}\subseteq\mathsf{SMN^{\circ}_t}$, for every $t\geq 2$. Suppose the contrary. Hence, there exists $S\in \mathsf{SMN}$, elements $y_1,\ldots,y_t,z_1,\ldots,z_t\in S$ and distinct integers $i$ and $j$ such that $\phi_t^{\omega}(y_i)\neq\phi_t^{\omega}(y_j)$. Therefore, we have
$$2\leq\abs{\{\lambda_{n,i}(y_1,\ldots,y_t;z_1,\ldots,z_t,z_1,\ldots)\mid 1\leq i\leq t\}},$$
for every positive integer $n$ which is a contradiction with $S\in \mathsf{SMN}$.

Now, suppose that there exists a finite semigroup $S$ which $S\not\in \mathsf{SMN}$ and $S\in(\bigcap_{2\leq t}\mathsf{SMN^{\circ}_t})$. If $S\not\in\mathsf{MN}$ then, by \cite[Theorem 3.1]{Al-Sha}, we have $S\not\in \mathsf{SMN}^{\circ}_2$, a contradiction. Hence, $S\in\mathsf{BG}_{nil}$ and $S\not\in \mathsf{SMN}$. By Lemma~\ref{[,]}, there exists an integer $t$ such that $S\not\in \mathsf{SMN}^{\circ}_t$, a contradiction.

The result follows.
\end{proof}

The following theorem shows that the pseudovariety $\mathsf{SMN}$ has infinite rank and, therefore, it is non-finitely based. 

\begin{thm} \label{MN' Rank}
The pseudovariety $\mathsf{SMN}$ has infinite rank.
\end{thm}

\begin{proof}
We prove that for every prime number $p$, there exists a finite semigroup $S$ such that $S$ is generated by $2p$ elements, $S\not\in\mathsf{SMN}$ and $$\langle x_1,\ldots, x_{2p-1}\rangle \in\mathsf{SMN},$$ for all $x_1,\ldots, x_{2t-1}\in S$.

Let the sets $A_p=\{\alpha_1,\ldots,\alpha_p\}$ and $B_p=\{\beta_1,\ldots,\beta_p\}$ with $A_p\cap B_p=\emptyset$ and the partial bijections
$X_{p,i}=(\alpha_i,\beta_i,\theta)$ and 
$$W_{p,i}=(\beta_1,\alpha_{1+i\pmod p},\theta)\,\cdots\,(\beta_p,\alpha_{p+i\pmod p},\theta)\quad (1\leq i\leq p)$$ 
on the set $A_p\cup B_p\cup\{\theta\}$.
Let $S_{p}$ be a subsemigroup of the full transformation semigroup on the set $A_p\cup B_p\cup\{\theta\}$ given by
$$
S_p=\langle X_{p,1},\ldots,X_{p,p},W_{p,1},\ldots,W_{p,p}\rangle.
$$
By Lemma~\ref{[,]}, the semigroup $S_{p}$ is not in $\mathsf{SMN}$. 

Suppose that a subsemigroup $T=\langle y_1,\ldots, y_{2p-1}\rangle$ of $S_p$ is not in $\mathsf{SMN}$. Since $S_{p}=\mathcal{M}^{0}(\{1\},2p,2p;I_{2p})\cup \{W_{p,1},\ldots,W_{p,p}\}$ and $T\not\in\mathsf{SMN}$, by Lemma~\ref{[,]}, there exist an integer $p'\leq p$ and elements 
$$a_i=(1;\alpha_{j_i},\beta_{j_i})\in\mathcal{M}^{0}(\{1\},2p,2p;I_{2p}),b_i\in\{W_{p,1},\ldots,W_{p,p}\}$$ such that $\abs{\{\alpha_{j_1},\ldots,\alpha_{j_{p'}}\}}=\abs{\{\beta_{j_1},\ldots,\beta_{j_{p'}}\}}=p'$ and
$$a_i=\lambda_{p',i}(a_1,\ldots,a_{p'};b_{1},\ldots, b_{p'}),$$ for every $1\leq i\leq p'$.
There exists an integer $1\leq k\leq p$ such that $b_1=W_{p,k}$ and, thus, we have $j_i+k=j_{i+1}\pmod p$, for every $1\leq i<p'$ and $j_{p'}+k=j_{1}\pmod p$. First, suppose that $p'$ is even. Then, we have $j_1+(p'/2)k=j_{1+p'/2}, j_1-(p'/2)k=j_{1+p'/2}\pmod p$ and, thus
$2j_1=2j_{1+p'/2}\pmod p$. Since the integers $j_1$ and $j_{1+p'/2}$ are distinct, we have $2\mid p$. As $p$ is prime, it follows that $p=p'=2$. Now, suppose that $p'$ is odd. Hence, we have $j_i+(l)k=j_{i+l\mod p'}, j_i-(l)k=j_{i-l\mod p'}\pmod p$, for every $1\leq i\leq p'$ and $1\leq l\leq (p'-1)/2$.
It follows that $j_1+j_2+\cdots+j_{p'}=p'j_1=\cdots=p'j_{p'}$. Hence, it follows that $p'\mid p$ and, thus, $p=p'$.
Therefore, we have $\{W_{p,1},\ldots,W_{p,p}\}\subsetneqq \{y_1,\ldots, y_{2p-1}\}$. Hence, there exists an integer 
$1\leq i\leq p$ such that $(\alpha_i,z,\theta)\not\in \{y_1,\ldots, y_{2p-1}\}$, for every $z\in A_p\cup B_p$ and, thus, $(\alpha_i,\beta_i,\theta)\not\in M$. Since $\{b_{1},\ldots, b_{p}\}=\{W_{p,1},\ldots,W_{p,p}\}$, we have $\{a_1,\ldots,a_{p}\}=\{(1;\alpha_i,\beta_i),\ldots,(1;\alpha_p,\beta_p)\}$, a contradiction.

The result follows.
\end{proof}

The following proposition can be seen as criteria to determine when a semigroup $S\in\mathsf{BI}$ is not Mal'cev nilpotent or is not strongly Mal'cev nilpotent. 

\begin{prop}\label{v-1v-2}
Let $S$ be a semigroup in the pseudovariety $\mathsf{BI}$. Suppose that there exist ideals $A,B$ of $S$, an inverse Rees matrix semigroup $M=\mathcal{M}^{0}(\{1\},q,q;I_q)$, an integer $t$ and elements $y_1,\ldots,y_t,v_1,v_2$ such that the following conditions are satisfied:
\begin{enumerate}
\item $B\subsetneqq A$ and $A/B\cong M$;
\item $1<t$;
\item $y_i=(1;\alpha_i,\beta_i)\in M$, for every $1\leq i\leq t$ and $\abs{\{\alpha_1,\ldots,\alpha_t\}}=t$;
\item $[\beta_1,\alpha_{1};\ldots;\beta_t,\alpha_{t}]\sqsubseteq \Gamma(v_1)$ and $[\beta_1,\alpha_{1+i\pmod t};\ldots;\beta_t,\alpha_{t+i\pmod t}]\sqsubseteq \Gamma(v_2),$ for some integer $1\leq i< t$, where $\Gamma$ is an ${\mathcal L}_M$-representation of $S/B$.
\end{enumerate}
Then, the elements $v_1$ and $v_2$ are not regular.
\end{prop}

\begin{proof}
We prove this by contradiction. Suppose that $v_2$ is regular.
Let $k=t/\mbox{gcd}(t,i)$.
Since $v_2$ is regular and $S\in \mathsf{BI}$, $v_2$ has an inverse element. Hence, we have
\begin{align*}
(\alpha_1,\alpha_{1+i\pmod t},\alpha_{1+2i\pmod t},\ldots,\alpha_{1-i\pmod t}&)\subseteq\Gamma((v_2^{-1}v_1)^{k-1}).
\end{align*}
Since $i<t$, we have $1<k$ and, thus, $S$ is not aperiodic.
This contradicts the assumption that $S\in \mathsf{BI}$.  

Similarly, we have a contradiction when $v_1$ is regular.
\end{proof}

The following example is presented to illustrate the determination of some strongly Mal'cev nilpotent semigroups by Proposition~\ref{v-1v-2}.

\begin{example}
Let $S$ be the subsemigroup of the full transformation semigroup on the set $\{1,\ldots,18\} \cup \{\theta\}$ such that 
$$S = \langle y_1,y_2,y_3,z_1,z_2,z_3\rangle,$$
where 
\begin{align*}
y_1=&(1,2,0)(13,14,0)(15,16,0),\\
y_2=&(5,6,0)(7,8,0)(17,18,0),\\
y_3=&(3,4,0)(9,10,0)(11,12,0),\\
z_1=&(2,7,0)(4,15,0)(6,11,0)(8,9,0)(10,1,0)(12,13,0)(14,5,0)(16,17,0)\\&(18,3,0),\\
z_2=&(2,3,0)(4,5,0)(6,1,0),\\
z_3=&(2,1,0)(4,3,0)(6,5,0)(8,7,0)(10,9,0)(12,11,0)(14,13,0)(16,15,0)\\&(18,17,0).
\end{align*}
The semigroup $S$ is strongly Mal'cev nilpotent.
\end{example}

\begin{proof}
The semigroup $S$ is aperiodic and $S$ has the principal series $$S= S_1 \supset S_2 \supset S_3 \supset S_{4} \supset S_{5}\supset S_{6} = \{\theta\}$$
where $S_5/S_6=(M_1=)\mathcal{M}^0(\{1\},18,18;I_{18})$,  $S_4/S_5=(M_2=)\mathcal{M}^0(\{1\},6,6;I_{6})$, $S_3\setminus S_4=\{z_1\}$, $S_2\setminus S_3=\{z_2\}$ and $S_1\setminus S_2=\{z_3\}$. Hence, only the elements $z_1,z_2$ and $z_3$ are non regular elements of $S$.
Let $\Gamma_i$ be an $M_i$-representation of $M_i$, for $1\leq i\leq 2$. We also have $\Gamma_2(z_2)=\theta$.

Since $S\in\mathsf{BI}$, if $S\not\in \mathsf{SMN}$, by Lemma~\ref{[,]} and Proposition \ref{v-1v-2}, one of the following conditions holds:
\begin{enumerate}
\item there exist distinct elements $a_1,a_2\in S$ and distinct elements 
$w_{1}, w_{2}\in\{z_1,z_2,z_3\}$ such that $a_i = \lambda_{2,i}(a_1,a_2;w_{1},w_{2})$, for all $1\leq i\leq 2$;
\item there exist pairwise distinct elements $a_1,a_2,a_3\in S$ and pairwise distinct elements 
$w_{1}, w_{2},w_3\in\{z_1,z_2,z_3\}$ such that $$a_i = \lambda_{3,i}(a_1,a_2,a_3;w_{1},w_{2},w_3),$$ for all $1\leq i\leq 3$.
\end{enumerate}
By using a Mathematica package developed by the first author, based on Proposition~\ref{Alg4}, one can check that $S\in\mathsf{MN}$, and it follows that the part (1) does not imply. Since $\abs{\{i\mid 1\leq i\leq 18\mbox{ and }\Gamma_1(z_2)(i)\neq\theta\}}=3$, $\Gamma_1(z_2)(2)\neq 0,(2,7,0)\subseteq\Gamma(z_1)$ and there does not exist any integer $i$ such that $\Gamma_1(z_2)(i)=7$, we have $a_1,a_2,a_3\not\in M_1\setminus\{\theta\}$ and, thus $a_1,a_2,a_3\in M_2\setminus\{\theta\}$. Now, as $\Gamma_2(z_2)=\theta$, the part (2) does not imply. A contradiction and, thus, $S\in \mathsf{SMN}$.
\end{proof}

%

We have $\langle\mathsf{A}\cap\mathsf{Inv}\rangle \subsetneqq\mathsf{A}\cap\mathsf{MN}$ (\cite[Theorem 8.1]{Al-Sha}). The following theorem presents the similar result for strong Mal'cev nilpotency. 

\begin{thm}\label{A MN'}
We have $\langle\mathsf{A}\cap\mathsf{Inv}\rangle \subsetneqq\mathsf{A}\cap\mathsf{SMN}$.
\end{thm}

\begin{proof}

Suppose that $S\in(\mathsf{A}\cap\mathsf{Inv}) \setminus(\mathsf{A}\cap\mathsf{SMN})$. 
Since $S\in\mathsf{BI}$ and $S\not\in \mathsf{SMN}$, by Lemma~\ref{[,]} and Proposition~\ref{v-1v-2}, $S$ is not inverse, a contradiction.
Hence, $\mathsf{A}\cap\mathsf{Inv}$ is contained in $\mathsf{SMN}$ and, thus, $\langle\mathsf{A}\cap\mathsf{Inv}\rangle \subseteq\mathsf{A}\cap\mathsf{SMN}$. 

By Lemma~\ref{[,]}, the semigroup $N_4$ in \cite{Al-Sha} is in the subset $\mathsf{A}\cap\mathsf{SMN}\setminus \langle\mathsf{A}\cap\mathsf{Inv}\rangle$.
Therefore, $\langle\mathsf{A}\cap\mathsf{Inv}\rangle$ is strictly contained in $\mathsf{A}\cap\mathsf{SMN}$.
\end{proof}

Note that, we can improve the result of Theorem~\ref{A MN'} and claim that $$\langle\mathsf{A}\cap\mathsf{Inv}\rangle \subsetneqq\langle\mathsf{Inv}\rangle\cap\mathsf{A}\cap\mathsf{SMN}.$$ 
Before we present an example to show that $\langle\mathsf{A}\cap\mathsf{Inv}\rangle$ is strictly contained in $\langle\mathsf{Inv}\rangle\cap\mathsf{A}\cap\mathsf{SMN}$, we recall some definitions from \cite{Hig-Mar}.
Consider the sets $X_n=\{1,\ldots,n\}$, $X'_n=\{1',\ldots,n'\}$, $X_{2n} = \{1, \ldots, n, 1', \ldots, n'\}$ and the Rees matrix semigroup $M=\mathcal{M}^{0}(\{1\},2n,2n;I_{2n})$. Take the action $\Gamma$ of the symmetric inverse semigroup $\mathcal{I}_{2n}$ on the ${\mathcal L}$-classes of $M$.
Let $b_1,b_2,\ldots ,b_k$ be bijections on the set $X_n$, and let $U$ be the semigroup generated by the bijections $b_i$, ($1\leq i \leq k$). Higgins and Margolis introduced
the subsemigroup $S(U)$ of $\mathcal{I}_{2n}$ as follows. For each $1 \leq i \leq k$, let $b'_i$ be the map with $\mbox{dom~} b'_i =\mbox{dom~} b_i$, and $\mbox{ran~} b'_i\subseteq (\mbox{ran~} b_i)'$ which acts as follows:
$\Gamma (b'_i)(\alpha) = (\Gamma (b_i)(\alpha))'$ for every $\alpha \in \mbox{dom~} b_i$.
Similarly, let $a'=(1,1',\theta)(2,2',\theta)\ldots(n,n',\theta) $.
Finally, let $S(U)$ be the semigroup generated by the mappings $b'_i$, ($1\leq i \leq k$), together with $a'$ and $M$. They prove that if $S(U)$ is a divisor of some finite inverse semigroup $I$, then $U$ divides $I$ also \cite[Theorem 3.2]{Hig-Mar}.

Now, we present our candidate to show that $\langle\mathsf{A}\cap\mathsf{Inv}\rangle \neq\langle\mathsf{Inv}\rangle\cap\mathsf{A}\cap\mathsf{SMN}$.
Let $S$ be the subsemigroup of the full transformation semigroup on the set $\{1,\ldots,6\} \cup \{\theta\}$ given by
the union of the completely $0$-simple semigroup $\mathcal{M}^0(\{1\},6,6;I_{6})$ and the set $\{w,v\}$,
$$
S=\mathcal{M}^0(\{1\},6,6;I_{6}) \cup \{ w,v\}
,$$
where $w=(1,5,\theta)(2,6,\theta)(3,4,\theta)$ and $v=(1,4,\theta)(2,5,\theta)(3,6,\theta)$. Thanks to Lemma~\ref{[,]}, $S$ is strongly Mal'cev nilpotent. Also, since the idempotents of $S$ commute, we have $S\in\langle\mathsf{Inv}\rangle$~\cite{Ash}.
We have $S=S(U)$ when $U$ is the semigroup generated by the elements
 $w'=(1,2,3)$ and $v'=(1)(2)(3)$. Now, if $S \prec I$, for some finite inverse semigroup $I$, then $U \prec I$. Since $U$ is not aperiodic, $I$ is not aperiodic and, thus, $S\not\in\langle\mathsf{A}\cap\mathsf{Inv}\rangle$.

The following propositions can be seen as criteria to determine when $S(U)$ is not Mal'cev nilpotent or is not strongly Mal'cev nilpotent.

\begin{prop} 
If there exist integers $i_1,i_2$ and a bijection $g\in \{b_1,b_2,\ldots ,b_k\}$ such that $(i_1,i_2)\subseteq g$, then the semigroup $S(U)$ is not Mal'cev nilpotent.
\end{prop}

\begin{proof}
We have $(i_1,i'_{2},\theta)(i_2,i'_{1},\theta)\subseteq g'$. Let $x_1=(i_1',i_2,\theta)$ and $x_2=(i_2',i_1,\theta)$. 
Now, we have 
$x_i = \lambda_{2,i}(x_1,x_1;a',g')$ for every $1\leq i\leq 2$. Then by Lemma~\ref{finite-nilpotent}, the semigroup $S(U)$ is not Mal'cev nilpotent.
\end{proof}

\begin{prop}\label{notMN'}
If there exist integers $i_1,\ldots,i_m$ with $m>1$ and a bijection $g\in \{b_1,b_2,\ldots ,b_k\}$ such that $(i_1,\ldots,i_m)\subseteq g$ and $g^2,\ldots,{g}^{m-1}\in \{b_1,b_2,\ldots ,b_k\}$, then the semigroup $S(U)$ is not strongly Mal'cev nilpotent.
\end{prop}

\begin{proof}
We have $$(i_1,i'_{r+1\pmod m},\theta)\,\cdots\,(i_m,i'_{r+m\pmod m},\theta)\subseteq (g^r)'$$ for every integer $1\leq r\leq m-1$. Let $x_j=(i_j',i_{j+1},\theta)$ for every $1\leq j\leq m-1$ and $x_m=(i_m',i_1,\theta)$. 
Now, we have 
$x_j = \lambda_{m,j}(x_1,\ldots,x_t;a',g',\ldots, (g^{m-1})')$ for every $1\leq j\leq m$. Then, by Lemma~\ref{finite-nilpotent'}, the semigroup $S(U)$ is not strongly Mal'cev nilpotent.
\end{proof}

\section{Bases of $\kappa$-identities within $\mathsf{BG}_{nil}$}

Let $S$ be a semigroup. We define  Property $\mathcal{P}_2$ for $S$ as follows:
\begin{quotation}
\noindent
if $y_1$ and $y_2$ are in a $\mathcal{J}$-class of $S$ and there exist elements $z_1,z_2\in S$ such that $y_iz_jy_{i+j\pmod 2}$ is in the $\mathcal{J}$-class of $y_1$ and $y_2$, for all $1\leq i,j\leq 2$,
then $y_1\mathcal{H}y_2$.
\end{quotation}

\begin{lem}\label{P_2}
Let $S\in \mathsf{BG}_{nil}$. The semigroup $S$ is $\mathsf{MN}$ if and only if $S$ satisfies Property $\mathcal{P}_2$.
\end{lem}

\begin{proof}
Suppose that $S$ does not satisfy Property $\mathcal{P}_2$. Then, there exist elements $y_1,y_2,z_1,z_2$ and a $\mathcal{J}$-class $J$ of $S$ such that $$y_1,y_2,y_iz_jy_{i+j\pmod 2}\in J,$$ for all $1\leq i,j\leq 2$, and $y_1$ and $y_2$ are not in the same $\mathcal{H}$-class. Since $y_1z_1y_{2}\in J$, we have $y_1z_1\in J$ and, thus, $J$ is a regular $\mathcal{J}$-class. As $S\in \mathsf{BG}_{nil}$, there exist ideals $A,B$ of $S$ and an inverse Rees matrix semigroup $M=\mathcal{M}^{0}(G,n,n;I_n)$ such that $B\subsetneqq A$, $A/B\cong M$ and $J=M\setminus \{\theta\}$.
Therefore, there exist elements $(g;\alpha,\beta),(g';\alpha',\beta')\in M$ such that $y_1=(g;\alpha,\beta)$ and $y_2=(g';\alpha',\beta')$. Since $y_1$ and $y_2$ are in different $\mathcal{H}$-classes, we have $\alpha\neq\alpha'$ or $\beta\neq\beta'$. As $y_iz_jy_{i+j\pmod 2}\in J,$ for all $1\leq i,j\leq 2$, we have $z_1,z_2\in S\setminus B$, $[\beta,\alpha';\beta',\alpha] \sqsubseteq \Gamma(z_1)$ and $[\beta',\alpha';\beta,\alpha] \sqsubseteq \Gamma(z_2)$, where $\Gamma$ is an ${\mathcal L}_M$-representation of $S/B$. Lemma~\ref{[,]'} entails that $S\not\in\mathsf{MN}$. 

Similarly, if $S\not\in\mathsf{MN}$, then, by Lemma~\ref{[,]'}, $S$ does not satisfy Property $\mathcal{P}_2$.
%
\end{proof}

We recall the canonical signature $\kappa$ which consists
of the basic multiplication operation and the unary operation $x^{\omega-1}$ (for more details see \cite{Alm3}).
Let $S$ be a semigroup. We define the $\kappa$-term 
$$\Delta(y_1,y_2;z_1,z_2)=((y_1z_2)^{\omega-1}y_1z_1(y_2z_2)^{\omega-1}y_2z_1)^{\omega}(y_1z_2)^{\omega},$$
for $y_1,y_2,z_1,z_2\in S$.
 
\begin{thm}\label{MN-omega1}
Let $\mathsf{MN^{\star}}=\llbracket \Delta(y_1,y_2;z_1,z_2)=\Delta(y_2,y_1;z_1,z_2)\rrbracket$.
We have $\mathsf{MN}=\mathsf{MN^{\star}}\cap \mathsf{BG}_{nil}$.
\end{thm}

\begin{proof}
First, we prove that $\mathsf{MN}\subseteq\mathsf{MN^{\star}}\cap \mathsf{BG}_{nil}$. Suppose the contrary. Since $\mathsf{MN}\subsetneqq\mathsf{BG}_{nil}$, there exist $S\in \mathsf{MN}$ and elements $y_1,y_2,z_1,z_2\in S$ such that $$(\Delta_1=)\Delta(y_1,y_2;z_1,z_2)\neq\Delta(y_2,y_1;z_1,z_2)(=\Delta_2).$$ Let
$$y_1'=(r_1z_1r_2z_1)^{\omega}r_1 \mbox{ and } y_2'=(r_2z_1r_1z_1)^{\omega}r_2$$ where $r_1=(y_1z_2)^{\omega-1}y_1$ and $r_2=(y_2z_2)^{\omega-1}y_2$.
Since, $$y_1'=(r_1z_1r_2z_1)^{\omega}(r_1z_1r_2z_1)^{\omega}r_1,$$ there exist elements $a$ and $b$ in $S^1$ such that $y_1'=ay_2'b$. Similarly, there exist elements $a',b'\in S^1$ such that $y_2'=a'y_1'b'$. It follows that $y_1'\mathcal{J}y_2'$. Note that we have 
\begin{align*} \label{r1z2}
r_1=(y_1z_2)^{\omega-1}y_1=(y_1z_2)^{\omega}(y_1z_2)^{\omega-1}y_1=(y_1z_2)^{\omega-1}y_1z_2(y_1z_2)^{\omega-1}y_1
=r_1z_2r_1.
\end{align*}
Similarly, we have $r_2=r_2z_2r_2$.
Since 
\begin{align*}
y_1'=&(r_1z_1r_2z_1)^{\omega}(r_1z_1r_2z_1)^{\omega}(r_1z_1r_2z_1)^{\omega}r_1\\
=&(r_1z_1r_2z_1)^{\omega}r_1z_1(r_2z_1r_1z_1)^{\omega}r_2z_1(r_1z_1r_2z_1)^{\omega-1}r_1\\
=&y_1'z_1y_2'z_1(r_1z_1r_2z_1)^{\omega-1}r_1
\end{align*}
and
\begin{align*}
y_1'=&(r_1z_1r_2z_1)^{\omega}(r_1z_1r_2z_1)^{\omega}r_1=(r_1z_1r_2z_1)^{\omega}r_1z_1r_2z_1(r_1z_1r_2z_1)^{\omega-1}r_1\\
=&(r_1z_1r_2z_1)^{\omega}r_1z_2r_1z_1r_2z_1(r_1z_1r_2z_1)^{\omega-1}r_1=(r_1z_1r_2z_1)^{\omega}r_1z_2(r_1z_1r_2z_1)^{\omega}r_1\\
=&y_1'z_2y_1',
\end{align*}
the elements $y_1'z_1y_2'$ and $y_1'z_2y_1'$ are in the $\mathcal{J}$-class of $y_1'$ and $y_2'$. Similarly, the elements $y_2'z_1y_1'$ and $y_2'z_2y_2'$ are in the $\mathcal{J}$-class of $y_1'$ and $y_2'$.
We supposed that $S$ is Mal'cev nilpotent. Hence, by Lemma~\ref{P_2}, $S$ satisfies Property $\mathcal{P}_2$. It follows that that $y_1'$ and $y_2'$ are in the same $\mathcal{H}$-class. Since $r_1=r_1z_2r_1$ and $r_2=r_2z_2r_2$, $\Delta_1$ and $\Delta_2$ are in the $\mathcal{J}$-class of $y_1'$ and $y_2'$. As $S\in\mathsf{BG}$, $y_1',y_2'$ are in the same $\mathcal{H}$-class, $\Delta_1=y'_1z_2$ and $\Delta_2=y'_2z_2$, $\Delta_1$ and $\Delta_2$ are in the same $\mathcal{H}$-class too. The elements $\Delta_1$ and $\Delta_2$ are idempotents. 
It follows that $\Delta_1=\Delta_2$, a contradiction.

Now, suppose there exists a finite semigroup $S$ such that $S\in\mathsf{MN^{\star}}\cap \mathsf{BG}_{nil}$ and $S\not\in \mathsf{MN}$.
Lemma~\ref{P_2} yields that $S$ does not satisfy Property $\mathcal{P}_2$ and, thus, $S\not\in \mathsf{MN^{\star}}$. A contradiction.

The result follows.
\end{proof}

\section{Comparison with $\mathsf{J} \malcev \mathsf{G_{nil}}$}\label{JGnil}
In this section, we compare the pseudovarieties $\mathsf{MN}$, $\mathsf{SMN}$ and $\mathsf{J} \malcev \mathsf{G_{nil}}$ where $\mathsf{J}$ is the pseudovariety of all finite $J$-trivial monoids. 

Let $A$ be a finite set, $F(A)$ be the free group on $A$ and $H$ be a finitely generated subgroup of $F(A)$.
In the seminal paper~\cite{Sta}, Stallings associated to $H$ an inverse automaton $\mathcal A(H)$ which can be used to solve a number of algorithmic problems concerning $H$ including the membership problem.  Stallings, in fact, used a different language than that of inverse automata; the automata theoretic formulation is from~\cite{Mar01}. Let $\widetilde A=A\cup A^{-1}$ where $A^{-1}$ is a set of formal inverses of the elements of $A$.
An \emph{inverse automaton} $\mathcal A$ over $A$ is an $\widetilde A$-automaton with the property that there is at most one edge labeled by each letter leaving each vertex and if there is an edge $p\to q$ labeled by $a$, then there is an edge $q\to p$ labeled by $a^{-1}$.  Moreover, we require that there is a unique initial vertex, which is also the unique terminal vertex. The set of all reduced words accepted by a finite inverse automaton is a finitely generated subgroup of $F(A)$ called the \emph{fundamental group} of the automaton. 

The $\widetilde A$-automaton $\mathcal{A}(H)$ (the Stallings automaton associated with $H$) is the unique finite connected inverse automaton whose fundamental group is $H$ with the property that all vertices have out-degree at least $2$ except possibly the initial vertex (where we recall that there are both $A$ and $A^{-1}$-edges). One description of $\mathcal  A(H)$ is as follows.  Take the inverse automaton $\mathcal A'(H)$ with vertex set the coset space $F(A)/H$ and with edges of the form $Hg\xrightarrow{\,\,a\,\,} Hga$ for $a\in \widetilde A$; the initial and terminal vertices are both $H$.  Then $\mathcal A(H)$ is the subautomaton whose vertices are cosets $Hu$ with $u$ a reduced word that is a prefix of the reduced form of some element $w$ of $H$ and with all edges between such vertices; the coset $H$ is still both initial and final. 
Stallings presented an efficient algorithm to compute $\mathcal A(H)$ from any finite generating set of $H$ via a procedure known as folding. 
From the construction, it is apparent that there is an automaton morphism $\mathcal A(H_1)\to \mathcal A(H_2)$ if and only if $H_1\subseteq H_2$ for finitely generated subgroups $H_1$ and $H_2$.  Also, it is known that $H$ has finite index if and only if $\mathcal A(H)=\mathcal A'(H)$. Stallings also provided an algorithm to compute $\mathcal A(H_1\cap H_2)$ from $\mathcal A(H_1)$ and $\mathcal A(H_2)$ (note that intersections of finitely generated subgroups of free groups are finitely generated by Howson's theorem). 

Conversely, if $\mathcal A = (Q,A, \delta,i, i)$ is a reduced inverse automaton, one can effectively construct a basis of a finitely generated subgroup $H$ of $F(A)$ such that $\mathcal A =\mathcal A(H)$. First we compute a spanning tree $T$ of the graph $\mathcal A$. For each state $q$ of $\mathcal A$, there is a unique shortest path from $i$ to $q$ within $T$: we let $u_q$ be the label (in ${\widetilde A}^{\star}$) of this path. Let $p_j
\xrightarrow{\,\,a_j\,\,} q_j$ ($1 \leq j \leq k$) be the $A$-labeled edges of $\mathcal A$ which are not in $T$. For each $j$, let $y_j = u_{p_j}a_ju^{-1}_{q_j}\in {\widetilde A}^{\star}$, and let $H = \langle y_1,\ldots, y_k\rangle$. Then $\{y_1,\ldots, y_k\}$ is a basis for $H$ and $\mathcal A =\mathcal A(H)$.

For another subgroup $K$ of $F(A)$, if $H \subseteq K$, the automaton congruence $\sim_{H,K}$ on $\mathcal A(H)$ is defined by the morphism from $\mathcal A(H)$ into $\mathcal A(K)$. 
Suppose that, for each state $p$ of $\mathcal A(H)$, $u_p$ is a reduced word such that $1.u_p = p$, in $\mathcal A(H)$. Then two states $p$ and $q$ of $\mathcal A(H)$ are $\sim_{H,K}$-equivalent if and only if $u_{p} u^{-1}_{q}\in K$. 

Let $\mathsf{V}$ be a pseudovariety of groups.
The subgroup $H$ is $\mathsf{V}$-extendible if its automaton can be embedded into a complete automaton with transition group in $\mathsf{V}$.
Let $\sim$ be the intersection of the $\sim_{H,K}$, where the intersection runs over all clopen subgroups $K$ in the pro-$\mathsf{V}$ topology containing $H$. The automaton
congruence $\sim$ coincides with $\sim_{H,\mathrm{Cl}_{\mathsf{V}}(H)}$ on $\mathcal A(H)$. Let $\widetilde{H}$ be the subgroup of $F(A)$ such that $\mathcal A(\widetilde{H}) = A(H)/\sim$. The subgroup $\widetilde{H}$ is the least $\mathsf{V}$-extendible subgroup containing $H$, and $H$ is $\mathsf{V}$-extendible if and only if $\widetilde{H}= H$. Also, in general $H \subseteq \widetilde{H} \subseteq \mathrm{Cl}_{\mathsf{V}}(H)$ and since the congruences $\sim$ and $\sim_{H,\mathrm{Cl}_{\mathsf{V}}(H)}$ on $\mathcal A(H)$ coincide, $\mathcal A(\widetilde{H}) = \mathcal A(H)/\sim$ embeds in $\mathcal A(\mathrm{Cl}_{\mathsf{V}}(H))$. 
See~\cite{Sta,Mar01,Ste01} for details.

Margolis, Sapir and Weil presented a procedure to compute the Stallings automaton of the $p$-closure of a finitely generated subgroup of a free group (which is again finitely generated), for every prime integer $p$. To compute the $p$-closure of $H$, we compute a finite sequence of quotients of $\mathcal A(H)$, $$\mathcal A(H_{0,p}) = \mathcal A(H)/\sim_0,\ldots,\mathcal A(H_{n,p}) = \mathcal A(H)/\sim_n,$$ such that each $H_{i,p}$ is $p$-closed, the automaton congruence $\sim_{i+1}$ is contained in $\sim_i$ (that is $H_{i+1,p}\subseteq H_{i,p})$, and $H_{n,p}$ is the $p$-closure of $H$. 
They let $\sim_0$ be the universal, one-class congruence, so that $H_{0,p}$ is a free factor of $F(A)$. Let $0\leq i$. After $i$ iterations of the algorithm, we have computed the quotient $\mathcal A(H_{i,p}) = \mathcal A(H)/\sim_i$. Roughly speaking, for the $(i + 1)$st iteration of
the algorithm, they translate $H$ into a basis of $H_{i,p}$ and they ask whether $H$ is $p$-dense in $H_{i,p}$. If it is, $H_{i,p}$ is the closure of $H$; if not, we compute the $(\mathbb{Z}/p\mathbb{Z})$-closure of $H$ in $H_{i,p}$, or rather a free factor $H_{i+1,p}$ of that closure which contains $H$. Formally, they present the following process:
\begin{enumerate}
\item \emph{Computing a basis of $H_{i,p}$}. First we compute a basis for $H_{i,p}$. Let $A_i$ be a set in bijection with that basis. We let $\kappa_i: F(A_i)\rightarrow H_{i,p}  \subseteq F(A)$ be the natural one-to-one morphism onto $H_{i,p}$. We denote by $\sigma_i$ the natural morphism $\sigma_i: F(A_i) \rightarrow (\mathbb{Z}/p\mathbb{Z})^{A_i}$.
\item \emph{Translating $H$ into the basis of $H_{i,p}$}. Now we compute a basis of the subgroup $\kappa_i^{-1}(H)$ of $F(A_i)$. This is done by running the elements of the basis of $H$ in $\mathcal A(H_i)$ and noting down the edges traversed that are not in the chosen spanning tree.
\item \emph{Deciding the $p$-denseness of $H$ in $H_{i,p}$}. Let $\mathfrak{M}_p(\kappa_i^{-1}(H))$ be the $r\times\abs{A_i}$ matrix consisting of the row vectors $\sigma_i\kappa_i^{-1}(h_1),\ldots,\sigma_i\kappa_i^{-1}(h_r)$. $\kappa_i^{-1}(H)$ is $p$-dense in $F(A_i)$ if and only if $\mathfrak{M}_p(\kappa_i^{-1}(H))$ has rank $\abs{A_i}$.
Then we calculate the rank of the matrix to decide whether $\kappa_i^{-1}(H)$ is $p$-dense in $F(A_i)$, and to compute a basis of $\sigma_i\kappa_i^{-1}(H)$ if it is not $p$-dense.
\item \emph{Stop if $H$ is $p$-dense in $H_{i,p}$}. If $H$ is $p$-dense in $H_{i,p}$, the algorithm stops: we now know that the $p$-closure of $H$ is $H_{i,p}$.
\item \emph{Otherwise compute $H_{i+1,p}$}. We now assume that $\kappa_i^{-1}(H)$ is not $p$-dense in $F(A_i)$. The subset $\sigma^{-1}_i\sigma_i\kappa_i^{-1}(H)$ is the $(\mathbb{Z}/p\mathbb{Z})$-closure of $\kappa_i^{-1}(H)$ in $F(A_i)$ and it is properly contained in $F(A_i)$. Since $\kappa_i$ is a homomorphism
from $F(A_i)$ onto $H_{i,p}$, the subgroup $K = \kappa_i\sigma^{-1}_i\sigma_i\kappa_i^{-1}(H)$ is the $(\mathbb{Z}/p\mathbb{Z})$-closure of $H$ in $H_{i,p}$ and $K\neq H_{i,p}$. We define the automaton congruence $\sim_{i+1}$ on $\mathcal A(H)$ to be $\sim_{H,K}$, the congruence induced by the containment of $H$ into
$K$. In particular, the subgroup $H_{i+1,p}$ such that $\mathcal A(H_{i+1,p}) =\mathcal A(H)\sim_{i+1}$ is a free factor of $K$, and hence $H_{i+1,p}$ is $p$-closed. Moreover, we have $H \subseteq H_{i+1,p} \subseteq K \subsetneqq H_{i,p}$, and hence $H_{i+1,p}$ is properly contained in $H_{i,p}$ and $\sim_{i+1}$ is properly contained in $\sim_{i}$. The automaton congruence $\sim_{i+1}$ is computed as follows. If $r$ and $s$ are states of $\mathcal A(H)$, we have $r\sim_{i+1} s$ if and only if $u_r u_s^{-1}\in K$, that is, if and only if $u_r u_s^{-1}\in H_{i,p}$ and $\sigma_i\kappa_i^{-1}(u_r u_s^{-1})\in \sigma_i\kappa_i^{-1}(H)$. To verify whether $u_r u_s^{-1}\in H_{i,p}$, and to compute in that case $\kappa_i^{-1}(u_r u_s^{-1})$, we run the reduced word obtained from $u_r u_s^{-1}$ in the automaton $\mathcal A(H_{i,p})$ starting at $1$, 
we note down the edges traversed that are not in the chosen spanning tree of that automaton (as in Step 2), 
and we require that this path ends in $1$. Then $\sigma_i\kappa_i^{-1}(u_r u_s^{-1})$ is the image of that word in $(\mathbb{Z}/p\mathbb{Z})^A$. Now it suffices to verify whether the vector $\sigma_i\kappa_i^{-1}(u_r u_s^{-1})$ lies in the vector subspace $\sigma_i\kappa_i^{-1}(H)$. This can be done effectively, using the basis of $\sigma_i\kappa_i^{-1}(H)$ computed in Step 3.
\end{enumerate}
They also proved that the nil-closure of $H$ is the intersection over all primes $p$ of the $p$-closures of $H$~\cite[Corollary 4.1]{Mar01}.

Let $M$ be a finite $A$-generated monoid and $\mathsf{H}$ a pseudovariety of groups.
Steinberg in \cite[Theorem 7.4]{Ste01} proved that $M\in \mathsf{J} \malcev \mathsf{H}$ if and only if the following conditions are satisfied:
\begin{enumerate}
\item $Sch_A(X)$ is an $\mathsf{H}$-extendible inverse $A$-graph, for each regular $\mathcal{R}$-class $X$;
\item $Sch_A^{\rho}(Y)$ is an $\mathsf{H}$-extendible inverse $A$-graph, for each regular $\mathcal{L}$-class $Y$.
\end{enumerate}

Let $A=\{a,b\}$ and $l$ be a positive integer. We define $\widetilde
A$-automata $\mathcal A_l$, with $l+1$ states, and $\mathcal B_l$ and
$\mathcal C_l$, each with $l$~states by the diagrams in
Figure~\ref{fig:diagrams}.
\begin{figure}[ht]
  \begin{center}
    \begin{tikzpicture}[>=latex, shorten >=0pt, shorten <=0pt, scale=0.9]
      \draw (-2,5.75) node (name) {$\mathcal A_l:$};
      \draw (1,4.38) node (1) {$\alpha_1$};
      \draw (2.8,3.51) node [] (2) {$\alpha_2$};
      \draw (3.25,1.56) node (3) {$\alpha_3$};
      \draw (2,0) node (4) {$\alpha_4$};
      \draw (0,0) node (ll-2) {$\alpha_{l-2}$};
      \draw (-1.25,1.56) node (ll-1) {$\alpha_{l-1}$};
      \draw (-0.8,3.51) node (ll) {$\alpha_l$};
      \draw (-0.5,5.25) node [initial, initial by arrow, initial text={}, initial where=above,
                              accepting, accepting by arrow, accepting where=above]
                              (ll+1) {$\alpha_{l+1}$};
      \path (1) edge [->, bend left=10, thick] node[above] {$a,b$} (2)
            (2) edge [->, bend left=10, thick] node[right] {$a,b$} (3)
            (3) edge [->, bend left=10, thick] node[right] {$a,b$} (4)
            (4) edge [->, bend left=10, dashed, thick] node[below] {$a,b$} (ll-2)
            (ll-2) edge [->, bend left=10, thick] node[left] {$a,b$} (ll-1)
            (ll-1) edge [->, bend left=10, thick] node[left] {$a,b$} (ll)
            (ll) edge [->, bend left=10, thick] node[above] {$a$} (1)
            (ll+1) edge [->, bend left=10, thick] node[above] {$b$} (1);
    \end{tikzpicture}

    \begin{tikzpicture}[>=latex, shorten >=0pt, shorten <=0pt, scale=0.9]
      \draw (-1,5.0) node (name) {$\mathcal B_l:$};
      \draw (1,4.38) node [initial, initial by arrow, initial text={}, initial where=above,
                           accepting, accepting by arrow, accepting where=above]
                          (1) {$\beta_1$};
      \draw (2.8,3.51) node (2) {$\beta_2$};
      \draw (3.25,1.56) node (3) {$\beta_3$};
      \draw (2,0) node (4) {$\beta_4$};
      \draw (0,0) node (ll-2) {$\beta_{l-2}$};
      \draw (-1.25,1.56) node (ll-1) {$\beta_{l-1}$};
      \draw (-0.8,3.51) node (ll) {$\beta_l$};
      \path (1) edge [->, bend left=10, thick] node[above] {$a,b$} (2)
            (2) edge [->, bend left=10, thick] node[right] {$a,b$} (3)
            (3) edge [->, bend left=10, thick] node[right] {$a,b$} (4)
            (4) edge [->, bend left=10, dashed, thick] node[below] {$a,b$} (ll-2)
            (ll-2) edge [->, bend left=10, thick] node[left] {$a,b$} (ll-1)
            (ll-1) edge [->, bend left=10, thick] node[left] {$a,b$} (ll)
            (ll) edge [->, bend left=10, thick] node[above] {$a$} (1);
    \end{tikzpicture}
    \qquad
    \begin{tikzpicture}[>=latex, shorten >=0pt, shorten <=0pt, scale=0.9]
      \draw (-1,5.0) node (name) {$\mathcal C_l:$}; 
      \draw (1,4.38) node [initial, initial by arrow, initial text={}, initial where=above,
                           accepting, accepting by arrow, accepting where=above]
                          (1) {$\gamma_1$}; 
      \draw (2.8,3.51) node (2) {$\gamma_2$}; 
      \draw (3.25,1.56) node (3) {$\gamma_3$}; 
      \draw (2,0) node (4) {$\gamma_4$}; 
      \draw (0,0) node (ll-2) {$\gamma_{l-2}$};
      \draw (-1.25,1.56) node (ll-1) {$\gamma_{l-1}$}; 
      \draw (-0.8,3.51) node (ll) {$\gamma_l$};
      \path (1) edge [->, bend left=10, thick] node[above] {$a,b$} (2)
            (2) edge [->, bend left=10, thick] node[right] {$a,b$} (3)
            (3) edge [->, bend left=10, thick] node[right] {$a,b$} (4)
            (4) edge [->, bend left=10, dashed, thick] node[below] {$a,b$} (ll-2)
            (ll-2) edge [->, bend left=10, thick] node[left] {$a,b$} (ll-1)
            (ll-1) edge [->, bend left=10, thick] node[left] {$a,b$} (ll)
            (ll) edge [->, bend left=10, thick] node[above] {$a,b$} (1);
    \end{tikzpicture}
  \end{center}
  \caption{Diagrams of the automata $\mathcal A_l$, $\mathcal B_l$,
    and $\mathcal C_l$}
  \label{fig:diagrams}
\end{figure}
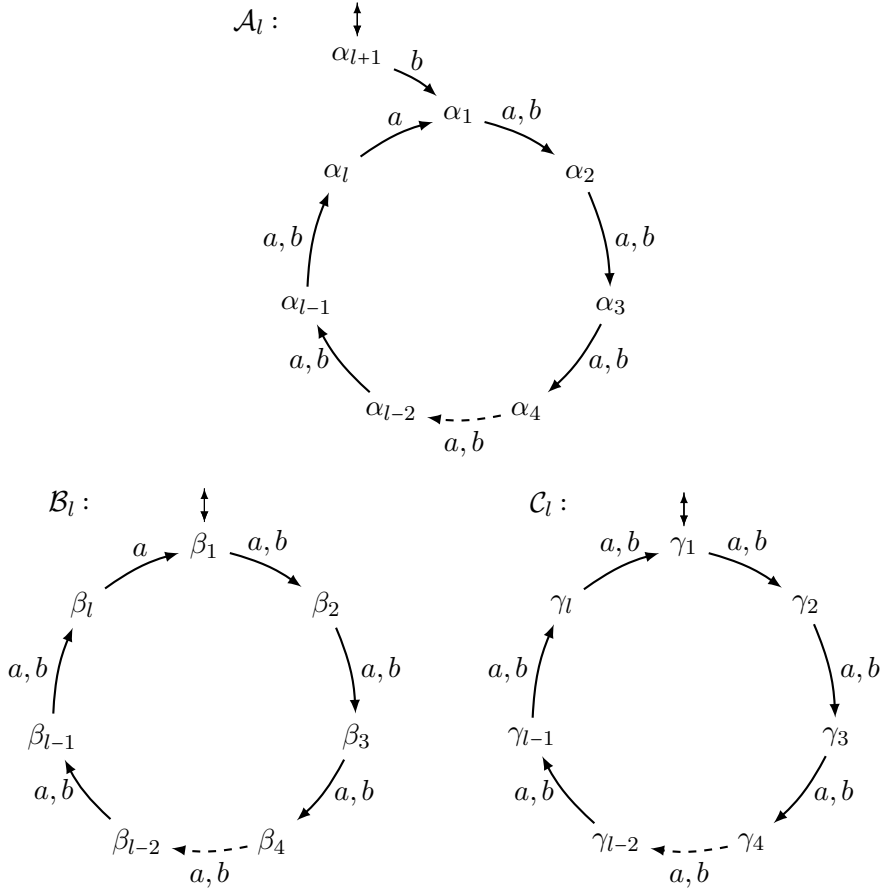
Margolis, Sapir and Weil proved that $\mathcal A_6$ is not $\mathsf{G}_{nil}$-extendible. We extend their result using a similar technique through the following lemma and theorem.

\begin{lem} \label{B_nC_n}
Let $n$ be a positive integer. Suppose that $n=p_1^{n_1}\,\cdots\, p_m^{n_m}$ where $p_1,\ldots,p_m$ are pairwise distinct prime numbers and $1\leq n_1,\ldots,n_m$. Let $\mathcal B_{n}=\mathcal A(H)$ and $\mathcal C_{n}=\mathcal A(H')$, for some finitely generated subgroups $H$ and $H'$ of $F(\{a,b\})$. If $m=1$, then $\mathrm{Cl}_{nil}(H)=H$, otherwise, we have $\mathrm{Cl}_{nil}(H)=H'$.
\end{lem}

\begin{proof}
We take as a spanning tree of $\mathcal A(H)$ the path from vertex $\beta_1$ to $\beta_n$, labeled $a$ for every edge. Then, we have
$$H=\langle ab^{-1},a^2b^{-1}a^{-1},\ldots,a^{n-1}b^{-1}a^{-(n-2)},a^n\rangle,$$ and we need to compute the rank of the matrix
$\begin{bmatrix}
       1 & -1\\[0.3em]
       n & 0
     \end{bmatrix}$. Since $n=p_1^{n_1}\,\cdots\, p_m^{n_m}$, for every prime $p$ with $p\not\in\{p_1,\ldots,p_m\}$, this matrix has rank 2 and, thus, $H$ is $p$-dense in $F(\{a,b\})$. 
    
Suppose that $p\in\{p_1,\ldots,p_m\}$. 
Since $\sigma_0(H)$ is generated by $a - b$, we have $\beta_i\sim_1 \beta_{i+p}\sim_1 \beta_{i+2p}\sim_1 \ldots$, for every integer $1\leq i\leq p$ and there is no relation $\sim_1$ between any vertices $\beta_{i}$ and $\beta_{i+k_1p+k_2}$, for every integer $1\leq i\leq p$, $0\leq k_1$ and $1\leq k_2\leq p-1$ with $i+k_1p+k_2\leq n$. 

If $n$ is a prime number, then $\mathcal{A}(H_{1,p})=\mathcal{B}_p$ and, thus $H$ is $p$-closed. Hence, we have $\mathrm{Cl}_{nil}(H)=H$.

Now, suppose that $n$ is not prime. First, we assume that $1<m$. 
There is no relation $\sim_1$ between any vertices $\beta_i$ and $\beta_j$ with $1\leq i,j\leq p$ and $i\neq j$. Since $1<m$, there are edges between $\beta_i$ and $\beta_{i+1}$ for $1\leq i\leq p-1$ labeled $a$ and $b$. Now, as $\beta_{i}\sim_1 \beta_{i+k_1p}$, for every integer $1\leq i\leq p$ and $0\leq k_1$, we have $\mathcal{A}(H_{1,p})=\mathcal{C}_p$.
Hence, we have 
\begin{align*}
H_{1,p}=\langle& ab^{-1}(=x_1),a^2b^{-1}a^{-1}(=x_2),\ldots,a^{p-1}b^{-1}a^{-(p-2)}(=x_{p-1}),a^{p-1}b(=x_p),\\
&a^p(=x_{p+1})\rangle.
\end{align*}
Let $1\leq n'\leq n-1$. There exist integers $0\leq k_1$ and $1\leq k_2\leq p$ such that $n'=k_1p+k_2$.
Since $a^{p}b^{-1}a^{-(p-1)}=x_{p+1}x_p^{-1}$ and
$$a^{n'}b^{-1}a^{-(n'-1)}=a^{k_1p}a^{k_2}b^{-1}a^{-(k_2-1)}a^{-k_1p},$$ we have $a^{n'}b^{-1}a^{-(n'-1)}=x_{p+1}^{k_1}x_ix_{p+1}^{-k_1}$, for some $1\leq i\leq p-1$, or $a^{n'}b^{-1}a^{-(n'-1)}=x_{p+1}^{k_1}x_{p+1}x_p^{-1}x_{p+1}^{-k_1}$.
Now, as $p\in\{p_1,\ldots,p_m\}$ and $$H=\langle ab^{-1},a^2b^{-1}a^{-1},\ldots,a^{n-1}b^{-1}a^{-(n-2)},a^n\rangle,$$ we have 
\begin{align*}
H=\langle& x_1,x_2,\ldots,x_{p-1},x_{p+1}x_p^{-1},x_{p+1}x_1x_{p+1}^{-1},\ldots,x_{p+1}x_{p-1}x_{p+1}^{-1},x_{p+1}x_{p+1}x_p^{-1}x_{p+1}^{-1},\\
&x_{p+1}^2x_1x_{p+1}^{-2},\dots,x_{p+1}^{r-1}x_{p-1}x_{p+1}^{-(r-1)},x_{p+1}^r\rangle
\end{align*}
which $r=n/p$. Then, we need to compute the rank of the matrix
$$\begin{bmatrix}
    1 & 0 &  \dots  & 0 & 0 & 0\\
    0 & 1 &  \dots  & 0 & 0 & 0\\
    \vdots & \vdots &  \ddots & \vdots & \vdots& \vdots\\
    0 & 0 &  \dots  & 1 & 0 & 0\\
    0 & 0 &  \dots  & 0 & -1 & 1\\
    1 & 0 &  \dots  & 0 & 0 & 0\\
    \vdots & \dots &  \vdots & \vdots & \vdots& \vdots\\
     0 & 0 &  \dots  & 1 & 0 & 0\\
    0 & 0 &  \dots  & 0 & 0 & r
     \end{bmatrix}.$$
If $n_p=1$, then the rank of this matrix is $p+1$ and, thus, $H$ is $p$-dense in $H_{1,p}$. Hence, we have $\mathrm{Cl}_{p}(H)=H_{1,p}$. 
Otherwise, the rank of this matrix is $p$ and, thus, we must calculate $H_{2,p}$. 
If $\gamma_i\sim_2 \gamma_j$, for some integers $1\leq i,j\leq p$, then $a^{i-j}\in H_{1,p}$ and, thus, $i-j=i'p$, for some integer $i'$. Hence, we have $u_iu_j^{-1}=x_{p+1}^{i'}$. If $\sigma_2(x_{p+1}^{i'})\in \sigma_2\kappa_2^{-1}(H)$, then $i'=i''p$, for some integer $i''$ and, thus $u_iu_j^{-1}=a^{i''p^2}$. It follows that
$\mathcal{A}(H_{2,p})=\mathcal{C}_{p^2}$. By induction, it is easy to verify that $\mathcal{A}(H_{i,p})=\mathcal{C}_{p^i}$ for all $1\leq i\leq n_p$. Then, we have 
\begin{align*}
H_{n_p,p}=\langle& ab^{-1}(=x_1),a^2b^{-1}a^{-1}(=x_2),\ldots,a^{p^{n_p}-1}b^{-1}a^{-(p^{n_p}-2)}(=x_{p^{n_p}-1}),\\
&a^{p^{n_p}-1}b(=x_{p^{n_p}}),a^{p^{n_p}}(=x_{p^{n_p}+1})\rangle
\end{align*}
and
$$H=\langle x_1,x_2,\ldots,x_{p^{n_p}-1},x_{p^{n_p}+1}x_{p^{n_p}}^{-1},\dots,x_{p^{n_p}+1}^{r_{n_p}}\rangle$$ which $r_{n_p}=n/p^{n_p}$. Then, we need to compute the rank of the matrix
$$\begin{bmatrix}
    1 & 0 &  \dots  & 0 & 0 & 0\\
    0 & 1 &  \dots  & 0 & 0 & 0\\
    \vdots & \vdots &  \ddots & \vdots & \vdots& \vdots\\
    0 & 0 &  \dots  & 1 & 0 & 0\\
    0 & 0 &  \dots  & 0 & -1 & 1\\
    1 & 0 &  \dots  & 0 & 0 & 0\\
    \vdots & \dots &  \vdots & \vdots & \vdots& \vdots\\
     0 & 0 &  \dots  & 0 & -1 & 1\\
    0 & 0 &  \dots  & 0 & 0 & r_{n_p}
     \end{bmatrix}.$$
The rank of this matrix is $p+1$ and, thus, $H$ is $p$-dense in $H_{n_p,p}$. Hence, we have $\mathrm{Cl}_{p}(H)=H_{n_p,p}$.

Therefore, we have $\mathrm{Cl}_{nil}(H)=H_{n_{p_1},p_1}\cap\ldots\cap H_{n_{p_m},p_m}$. 
Let $\lambda=\omega_1^{\kappa_1}\,\cdots\, \omega_{m'}^{\kappa_{m'}}$ with $\omega_1,\ldots,\omega_{m'}\in\{a,b\}$ and $\mathcal{C}_l=\mathcal{A}(H_{\mathcal{C}_l})$ for some integer $l$ and finitely generated subgroup $H_{\mathcal{C}_l}$ of $F(\{a,b\})$. We have $\lambda \in H_{\mathcal{C}_l}$ if and only if $\kappa_1+\cdots+\kappa_{m'} \equiv_l 0$. Hence $\lambda\in H_{n_{p_1},p_1}\cap\ldots\cap H_{n_{p_m},p_m}$ if and only if $\kappa_1+\cdots+\kappa_{m'} \equiv 0 \pmod {p_i^{n_i}}$, for all $1\leq i\leq m$. Therefore, we have $H_{n_{p_1},p_1}\cap\ldots\cap H_{n_{p_m},p_m}=H'$. 

Now, we assume that $m=1$. Similarly, we have $\mathcal{A}(H_{i,p})=\mathcal{C}_{p^{i}}$, for all $1\leq i\leq n_p-1$, and $\mathcal{A}(H_{n_p,p})=\mathcal{B}_{p^{n_p}}$. Thus $H$ is $p$-closed and, we have $\mathrm{Cl}_{nil}(H)=H$.
\end{proof}

\begin{thm} \label{A_n}
Let $\mathcal A$ be an inverse automaton. If there exists an integer $n$ such that $n=p_1^{n_1}\,\cdots\, p_m^{n_m}$ where $p_1,\ldots,p_m$ are pairwise distinct prime numbers, $1\leq n_1,\ldots,n_m$, $m>1$ and 
$\mathcal A_n$ is a subgraph of $\mathcal A$, then $\mathcal A$ is not $\mathsf{G}_{nil}$-extendible.
\end{thm}

\begin{proof}
First, we prove that the automaton $\mathcal A_n$ is not $\mathsf{G}_{nil}$-extendible.

There exist finitely generated subgroups $H$, $H'$ and $H''$ of $F(\{a,b\})$ such that $\mathcal A_n=\mathcal A(H)$, $\mathcal B_n=\mathcal A(H')$, and $\mathcal C_n=\mathcal A(H'')$. 
Since $M(\mathcal C_n)$ is a cyclic group, $H''$ is a normal subgroup of $F(\{a,b\})$ and, thus, $bH''b^{-1}=H''$. As conjugation by $b$ is a homomorphism, we have $\mathrm{Cl}_{nil}(H)=\mathrm{Cl}_{nil}(bH'b^{-1})=b\mathrm{Cl}_{nil}(H')b^{-1}$. 
By Lemma~\ref{B_nC_n}, it follows that $\mathrm{Cl}_{nil}(H')=H''$.
Now, as $bH''b^{-1}=H''$, we have $\mathrm{Cl}_{nil}(H)=H''$. If $\mathcal A_n$ is $\mathsf{G}_{nil}$-extendible, then 
$\mathcal A_n$ embeds in $\mathcal C_n$.
This yields a contradiction.

If the automaton $\mathcal A$ is $\mathsf{G}_{nil}$-extendible, then there is a complete automaton $\mathcal D$ such that $\mathcal A\subseteq \mathcal D$ and $M(\mathcal D)\in \mathsf{G}_{nil}$. Hence, there is a complete automaton $\mathcal D'$ such that $\mathcal A_n\subseteq \mathcal D'$ and $M(\mathcal D')\in \mathsf{G}_{nil}$ which is a contradiction.

The result follows.
\end{proof}

\begin{thm} \label{A_nMN}
Let $n$ be a positive integer and $N$ be the subsemigroup of the full transformation semigroup on the set $\{1,\ldots,n+1\} \cup \{\theta\}$ such that 
$$N = \mathcal{M}^0(\{1\},n+1,n+1;I_{n+1})  \cup \langle a,b\rangle\cup\{1\},$$
where $a=(1,2,\ldots,n)$ and $b=(n+1,1,2,\ldots,n,\theta)$.  Then, $N\in\mathsf{MN}$ if and only if the integer $n$ is odd.  
\end{thm}

\begin{proof}
If $n$ is an even integer, then we have $n=2n'$ for some positive integer $n'$. It follows that 
$$(1;1,n'+1) = \lambda_{2}((1;1,n'+1), (1;n'+1,1),a^{n},a^{n'})$$ and 
$$(1;n'+1,1) = \rho_{2}((1;1,n'+1), (1;n'+1,1),a^{n},a^{n'}).$$ Therefore, we have $N\not\in\mathsf{MN}$.

Now, we suppose that $n$ is odd. We prove $N\in\mathsf{MN}$ by contradiction. If $N\not\in\mathsf{MN}$, then, by Lemma~\ref{finite-nilpotent}, there exist a positive integer $m$, distinct elements $x, y\in N$ and elements 
$ w_{1}, w_{2}, \ldots, w_{m}\in N$ such that $x = \lambda_{m}(x, y, w_{1}, w_{2}, \ldots, w_{m})$ and $y = \rho_{m}(x,y, w_{1}, w_{2}, \ldots, w_{m})$. Since $N$ is the subsemigroup of the full transformation semigroup on the set $\{1,\ldots,n+1\} \cup \{\theta\}$, there exist integers $1\leq e_1,\ldots,e_{n_1}\leq n+1$ and $1\leq f_1,\ldots,f_{n_2}\leq n+1$ such that $\abs{\{e_1,\ldots,e_{n_1}\}}=n_1$, $\abs{\{f_1,\ldots,f_{n_2}\}}=n_2$,
$\Gamma(x)(e_i),\Gamma(y)(f_j)\neq \theta$, for every $1\leq i\leq n_1$ and $1\leq j\leq n_2$, 
and $\Gamma(x)(e)=\Gamma(y)(f)= \theta$, for every $e\not\in\{e_1,\ldots,e_{n_1}\}$ and $f\not\in\{f_1,\ldots,f_{n_2}\}$
where $\Gamma$ is an ${\mathcal L}_{\mathcal{M}^0(\{1\},n+1,n+1;I_{n+1})}$-representation of $N$. 
It is easy to verify that $x$ and $y$ are in a regular $\mathsf{J}$-class. Hence, we have $n_1=n_2$.

Since $1=(1)(2)\,\cdots\,(n+1)$, $a=(1,2,\ldots,n)$ and $b=(n+1,1,2,\ldots,n,\theta)$, the elements $1$ and $w$ are not in the same $\mathsf{J}$-class for every $w\in \langle a,b \rangle$. Hence, the $\mathsf{J}$-class of $1$ has only one element and, thus $x,y\neq 1$.  

First, suppose that $x,y\in \langle a,b\rangle$. Thus, we have $w_1,\ldots,w_m\in \langle a,b\rangle^1$ and there exist letters $$c_1,\ldots,c_{m_1},d_1,\ldots,d_{m_2},w_{1,1},\ldots,w_{1,l_1},\ldots,w_{m,1},\ldots,w_{m,l_m}\in\{a,b\}$$ such that $x=c_1\,\cdots\, c_{m_1}$, $y=d_1\,\cdots\, d_{m_2}$ and $w_i=w_{i,1}\,\cdots\, w_{i,l_i}$  (if $w_i=1$, we put $l_i=0$), for all $1\leq i\leq m$. Since $x,y\neq 0$, there exist integers $1\leq i,j\leq n+1$ and $1\leq i',j'\leq n$ such that $\Gamma(x)(i)=i'$ and $\Gamma(y)(j)=j'$.
As $x = \lambda_{m}(x, y, w_{1}, w_{2}, \ldots, w_{m})$, if $i\neq n+1$, then we have 
$$m_1=2^{m-1}(m_1+m_2+l_1)+2^{m-2}l_2+\ldots+2^0l_m=i'-i\pmod n;$$
otherwise, we have
$$m_1=2^{m-1}(m_1+m_2+l_1)+2^{m-2}l_2+\ldots+2^0l_m=i'\pmod n.$$
Similarly, as $y = \rho_{m}(x,y, w_{1}, w_{2}, \ldots, w_{m})$,  if $j\neq n+1$, then we have 
$$m_2=2^{m-1}(m_2+m_1+l_1)+2^{m-2}l_2+\ldots+2^0l_m=j'-j\pmod n;$$
otherwise, we have
$$m_2=2^{m-1}(m_2+m_1+l_1)+2^{m-2}l_2+\ldots+2^0l_m=j'\pmod n.$$ 
Therefore we have $m_1=m_2\pmod n$. 

Again, as $x = \lambda_{m}(x, y, w_{1}, w_{2}, \ldots, w_{m})$, if $1\leq i\leq n_1$, then $$\Gamma(y)(\Gamma(xw_1)(e_i))\neq 0$$ and, thus, $\Gamma(xw_1)(e_i)\in\{f_1,\ldots,f_{n_1}\}$. Similarly, as $y = \rho_{m}(x,y, w_{1}, \ldots, w_{m})$, if $1\leq j\leq n_1$, then $$\Gamma(x)(\Gamma(yw_1)(f_j))\neq 0$$ and, thus, $\Gamma(yw_1)(f_j)\in\{e_1,\ldots,e_{n_1}\}$.
Now, since $x\neq y$ and $m_1=m_2\pmod n$, there exist subsets $$\{i_1,\ldots,i_{n'}\}, \{j_1,\ldots,j_{n'}\}\subseteq\{1,\ldots,n_1\}$$ such that $e_{i_{t+1}}-e_{i_{t}}=2(m_1+l_1)\pmod n$, $f_{j_{t+1}}-f_{j_{t}}=2(m_1+l_1)\pmod n$, $f_{j_{t}}-e_{i_{t}}=(m_1+l_1)\pmod n$, for every $1\leq t< n'$, $e_{i_{1}}-e_{i_{n'}}=2(m_1+l_1)\pmod n$, $f_{j_{1}}-f_{j_{n'}}=2(m_1+l_1)\pmod n$, $f_{j_{n'}}-e_{i_{n'}}=(m_1+l_1)\pmod n$ and $\{i_1,\ldots,i_{n'}\}\neq \{j_1,\ldots,j_{n'}\}$. 

Since $n$ is odd, there exist integers $r$ and $s$ such that $2r+ns=1$. Hence, we have $2r(m_1+l_1)=(m_1+l_1)\pmod n$.
Therefore, $\{i_1,\ldots,i_{n'}\}\cap \{j_1,\ldots,j_{n'}\}\neq \emptyset$ and thus $\{i_1,\ldots,i_{n'}\}= \{j_1,\ldots,j_{n'}\}$, a contradiction. 

Now, suppose that $x,y\in \mathcal{M}^0(\{1\},n+1,n+1;I_{n+1})\setminus \langle a,b\rangle$.
It follows that $x=(1;\alpha_1,\beta_1)$ and $y=(1;\alpha_2,\beta_2)$, for some integers $1\leq \alpha,\beta\leq n+1$. Thus, we have $[\beta_1,\alpha_2;\beta_2,\alpha_1]\sqsubseteq\Gamma(w_1)$ and $[\beta_1,\alpha_1;\beta_2,\alpha_2]\sqsubseteq\Gamma(w_2)$. It is easy to verify that $w_1,w_2\in\langle a,b\rangle^1$. Hence, $\alpha_2-\beta_1=\alpha_1-\beta_2\pmod n$ and $\alpha_1-\beta_1=\alpha_2-\beta_2\pmod n$. It follows that $2(\alpha_2-\alpha_1)=0\pmod n$. Since $[\beta_1,\alpha_2;\beta_2,\alpha_1]\sqsubseteq\Gamma(w_1)$, we have $\alpha_1,\alpha_2\neq n+1$. As $n$ is odd, it follows that $\alpha_1=\alpha_2$. Now again, as $[\beta_1,\alpha_2;\beta_2,\alpha_1]\sqsubseteq\Gamma(w_1)$, we have $\beta_1=\beta_2$ and, thus, $x=y$. A contradiction.

The result follows.
\end{proof}

Let $N_1$ be the subsemigroup of the full transformation semigroup on the set $\{1,\ldots,7\} \cup \{\theta\}$ such that
$$N_1 = \mathcal{M}^0(\{1\},7,7;I_{7}) \cup \langle a_{1},b_{1}\rangle\cup\{1\},$$
where $a_{1}=(1,2,\ldots,6)$ and $b_{1}=(7,1,2,\ldots,6,\theta)$ 
and $N_2$ be the subsemigroup of the full transformation semigroup on the set $\{1,\ldots,16\} \cup \{\theta\}$ such that
$$N_2 = \mathcal{M}^0(\{1\},16,16;I_{16}) \cup \langle a_{2},b_{2}\rangle\cup\{1\},$$
where $a_{2}=(1,2,\ldots,15)$ and $b_{2}=(16,1,2,\ldots,15,\theta)$.
By Theorem~\ref{A_n} and \cite[Theorem 7.4]{Ste01}, we have $N_1,N_2\not\in \mathsf{J} \malcev \mathsf{G}_{nil}$. Using, for instance 
a Mathematica package developed by the first author, one can check that $N_1\in\mathsf{BG}_{nil}$. By Theorem~\ref{A_nMN}, it follows that $N_1\not\in\mathsf{MN}$, $N_2\in\mathsf{MN}$. Since $\mathcal{M}^0(\{1\},16,16;I_{16})$ is an ideal of $N_2$ and $(1,2,\ldots,15)\in N_2$, we have $N_2\not\in \mathsf{SMN}$.

\begin{opm} 
Does there exist a finite semigroup $S$ such that $S\in \mathsf{SMN}\setminus \mathsf{J} \malcev \mathsf{G}_{nil}$?
\end{opm} 

Now, we present semigroups $M_i$ such that $M_i\in \mathsf{BG}_{nil}$, for all $1\leq i\leq 3$, and the following conditions are satisfied:
\begin{enumerate}
\item $M_1\in \mathsf{J} \malcev \mathsf{G_{nil}}$ and $M_1\in\mathsf{SMN}$;
\item $M_2\in \mathsf{J} \malcev \mathsf{G_{nil}}$ and $M_2\in(\mathsf{MN}\setminus \mathsf{SMN})$;
\item $M_3\in \mathsf{J} \malcev \mathsf{G_{nil}}$ and $M_3\not\in\mathsf{MN}$.
\end{enumerate} 

Let $M_1,M_2$ be the subsemigroup of the full transformation semigroup on the set $\{1,\dots,6\} \cup \{\theta\}$ such that
$$M_1 = \mathcal{M}^0(\{1\},6,6;I_{6}) \cup \{ c_{1},d_{1},1\},$$
where $c_{1}=(1,4,\theta)(2,5,\theta)(3,6,\theta)$ and $d_{1}=(1,5,\theta)(2,6,\theta)(3,4,\theta)$,
$$M_2 = \mathcal{M}^0(\{1\},6,6;I_{6}) \cup \{ c_{2},d_{2},e_{2},1\},$$
where $c_{2}=(1,4,\theta)(2,5,\theta)(3,6,\theta)$, $d_{2}=(1,5,\theta)(2,6,\theta)(3,4,\theta)$ and $e_{2}=(1,6,\theta)(2,4,\theta)(3,5,\theta)$, and  $M_3$ be the subsemigroup of the full transformation semigroup on the set $\{1,\dots,4\} \cup \{\theta\}$ such that
$$M_3 = \mathcal{M}^0(\{1\},4,4;I_{4}) \cup \{ c_{3},d_{3},1\},$$
where $c_{3}=(1,2,\theta)(3,4,\theta)$ and $d_{3}=(1,4,\theta)(3,2,\theta)$. 

By Lemmas~\ref{[,]'} and \ref{[,]}, we have $M_1\in\mathsf{SMN}$, $M_2\in(\mathsf{MN}\setminus \mathsf{SMN})$ and $M_3\not\in\mathsf{MN}$.
Also, by \cite[Theorem 7.4]{Ste01}, we have $M_1,M_2,M_3\in \mathsf{J} \malcev \mathsf{G_{nil}}$.

In terms of pseudovarieties, the results of this section may be summarized as follows:
\begin{thm}\label{MN-JmGnil} 
The pseudovarieties $\mathsf{MN}$ and $\mathsf{J} \malcev \mathsf{G_{nil}}$ are incomparable.
\end{thm}

\textit{Acknowledgments.}
The work of the first and third authors was supported, in part, by CMUP
(UID/MAT/00144/2013), which is funded by FCT (Portugal)
with national (MCTES) and European structural funds through the programs
FEDER, under the partnership agreement PT2020. The second author was supported by the DFG projects DI 435/5-2 and KU 2716/1-1. The work of the third author was also partly supported by the FCT
post-doctoral scholarship SFRH/BPD/89812/2012. The work was also carried out within an FCT/DAAD bilateral collaboration project involving the Universities of Porto and Stuttgart.

\bibliographystyle{plain}
\bibliography{ref-SMN-LA}

\end{document}